\pdfoutput=1

\documentclass[12pt]{article}
\usepackage{graphicx} 

\usepackage{geometry}
\usepackage{inputenc}
\usepackage{enumerate}
\usepackage{amssymb}
\usepackage{amsmath}
\usepackage{mathrsfs}
\usepackage{amsthm}
\usepackage{tikz-cd}
\usepackage{comment}
\usepackage{setspace}
\usepackage{hyperref}

\theoremstyle{definition}
\newtheorem{definition}{Definition}[section]
\newtheorem{example}[definition]{Examples}
\newtheorem{remark}[definition]{Remark}
\theoremstyle{plain}
\newtheorem{theorem}[definition]{Theorem}
\newtheorem{lemma}[definition]{Lemma}
\newtheorem{proposition}[definition]{Proposition}
\newtheorem{corollary}[definition]{Corollary}

\DeclareMathOperator{\supp}{supp}
\DeclareMathOperator{\dom}{Dom}
\DeclareMathOperator{\ran}{Ran}

\title{The twisted coarse Baum--Connes conjecture and relative hyperbolic groups}
\author{Jintao Deng\thanks{Department of Mathematics, State University of New York at Buffalo, Buffalo, NY, 14260} \and Ryo Toyota\thanks{Department of Mathematics, Texas A\&M University, College Station, TX 77843, USA}}
\date{}

\begin{document}

\maketitle

\begin{abstract}
In this paper, we introduce a notion of stable coarse algebras for metric spaces with bounded geometry, and formulate the twisted coarse Baum--Connes conjecture with respect to stable coarse algebras. We establish permanence properties of the conjecture under coarse equivalences, as well as for unions and subspaces.

As an application, we study higher index theory for a group $G$ that is hyperbolic relative to 
a finite family of subgroups $\{H_1, H_2, \dots, H_N\}$. We prove that $G$ satisfies the twisted coarse Baum--Connes conjecture with respect to any stable coarse algebra if and only if each subgroup $H_i$ does. 
\end{abstract}

\tableofcontents

\section{Introduction}

Let $X$ be a metric space with bounded geometry. The coarse Baum-Connes conjecture claims that a certain coarse assembly map 
$$
\mu: KX_*(X)\to K_*(C^*(X))
$$
is an isomorphism, where the left-hand side is the coarse $K$-homology and the right-hand side is the $K$-theory of the Roe algebra of $X$. This conjecture has many important applications in geometry and topology, including the Gromov--Lawson conjecture on the existence of Riemannian metrics with positive scalar curvature, and the Novikov conjecture on the homotopy invariance of higher signature on oriented closed manifolds. It has been verified for a large class of metric spaces, for example, metric spaces which admit a coarse embedding into Hilbert space \cite{Yu2000CoarseBaumConnesforSpacesUniformlyEmbeddable}, and some metric spaces which are not coarsely embeddable into Hilbert space, including relative expanders \cite{Deng-Wang-Yu:CBC-for-relative-expanders}.

To study the coarse Baum--Connes conjecture for the extension $1\to N \to G\to G/N \to 1$ with $N$ and $G/N$ coarsely embeddable into Hilbert space (cf.~\cite{Arzhantseva-Tessera:Extension-Coarse-Embedding}), the first author and L.~Guo \cite{DengGuo2024TwistedRoeAlgebras} developed a twisted coarse Baum-Connes conjecture with coefficients. In their work, they introduced the notion of a coarse algebra. Roughly speaking, for a metric space $X$ with bounded geometry, a $C^*$-algebra $\mathcal{A}$ and the algebra of all compact operators $\mathcal{K}$ on a separable infinite dimensional Hilbert space $\mathcal{H}$, a coarse algebra $\Gamma(X, \mathcal{A}\otimes \mathcal{K})$ is a $C^*$-subalgebra of $\ell^{\infty}(X, \mathcal{A}\otimes\mathcal{K})$ which is invariant under pullbacks by partial translations. This is a coarse geometric counterpart of a coefficient of a crossed product $C^*$-algebra. Then they formulated the twisted Roe algebra $C^*(X;\Gamma(X,\mathcal{A}\otimes\mathcal{K}))$ associated with a coarse algebra $\Gamma(X,\mathcal{A}\otimes\mathcal{K})$. Using this framework, they verified that the coarse Baum-Connes conjecture holds for a metric space with a coarse fibration for which each fiber and the base space are coarsely embeddable into Hilbert space, and this applies to extensions of coarse embeddable groups.

Based on the aforementioned framework of the twisted coarse Baum--Connes conjecture, we introduce the notion of stable coarse algebra for a metric space with bounded geometry and define a refined version of twisted coarse Baum--Connes conjecture. This refined version of twisted coarse Baum--Connes conjecture has better permanence properties, such as inheritance to subspaces and invariance under coarse equivalence.
We say that a coarse algebra $\Gamma(X, \mathcal{A}\otimes \mathcal{K})$ is stable if it is invariant under multiplication in the $B(\mathcal{H})$-direction and under forming a matrix (Definition~\ref{stable coarse algebra}). 
We can generalize the twisted Roe algebra to non-discrete metric spaces with bounded geometry. Let $M$ be a metric space and $X\subset M$ a net in $M$. Given a stable coarse $X$-algebra $\Gamma(X, \mathcal{A}\otimes \mathcal{K})$, we can define the $\Gamma(X, \mathcal{A}\otimes \mathcal{K})$-twisted Roe algebra $C^*(M; \Gamma(X, \mathcal{A}\otimes \mathcal{K}))$.

In this paper, we study twisted Roe algebras $C^*(M;\Gamma(X,\mathcal{A}\otimes \mathcal{K}))$ associated with a stable coarse algebra $\Gamma(X,\mathcal{A}\otimes\mathcal{K})$. 

\begin{theorem}
    Let $M$ and $N$ be metric spaces with bounded geometry with fixed nets $X\subset M$ and $Y\subset N$, respectively, and let $f: M\to N$ be a coarse equivalence. For any stable coarse $Y$-algebra $\Gamma(Y, \mathcal{A}\otimes \mathcal{K})$, there exists a stable coarse $X$-algebra $\Gamma(X, \mathcal{A}\otimes \mathcal{K})$ such that 
    $$
    C^*(M; \Gamma(X, \mathcal{A}\otimes \mathcal{K}))\cong
    C^*(N; \Gamma(Y, \mathcal{A}\otimes \mathcal{K})).
    $$
\end{theorem}

Let $X$ be a discrete metric space with bounded geometry, and $\Gamma(X, \mathcal{A}\otimes \mathcal{K})$ a coarse $X$-algebra. For each Rips complex $P_s(X)$, we can define the $\Gamma(X, \mathcal{A}\otimes \mathcal{K})$-twisted localization algebra, denoted by $C^*_L(P_s(X), \Gamma(X, \mathcal{A}\otimes \mathcal{K}))$. This is a $C^*$-subalgebra of the localization algebra with coefficients in $\mathcal{A}$. There is a $*$-homomorphism 
$$e_*: K_*(C_L^*(P_s(X);\Gamma(X,\mathcal{A}\otimes \mathcal{K})))\to K_*(C^*(P_s(X);\Gamma(X, \mathcal{A}\otimes \mathcal{K})))\cong K_*(C^*(X;\Gamma(X, \mathcal{A}\otimes \mathcal{K})))$$
induced by the evaluation.
Passing to infinity, we obtain the $\Gamma(X,\mathcal{A}\otimes \mathcal{K})$-twisted coarse assembly map
$$
 \mu: \varinjlim_s K_*(C_L^*(P_s(X);\Gamma(X,\mathcal{A}\otimes \mathcal{K})))\to K_*(C^*(X;\Gamma(X, \mathcal{A}\otimes \mathcal{K}))).
$$
The twisted coarse Baum--Connes conjecture \cite{DengGuo2024TwistedRoeAlgebras} claims that the twisted coarse assembly map is an isomorphism with respect to any coarse $X$-algebra $\Gamma(X, \mathcal{A}\otimes \mathcal{K})$. We study when it is an isomorphism for any stable coarse algebras.

The idea of the twisted coarse Baum--Connes conjecture originated from G.~Yu's proof of coarse Baum--Connes conjecture for metric spaces with bounded geometry admitting a coarse embedding into Hilbert space  \cite{Yu2000CoarseBaumConnesforSpacesUniformlyEmbeddable}.


The twisted coarse Baum--Connes conjecture captures more information about subspaces. We prove the following coarse analogue of the permanence property of Baum--Connes conjecture with coefficients under taking subgroups \cite{ChabertEchterhoff2001PermanencePropertiesofBC}. 

\begin{theorem}
    Let $X$ be a metric space with bounded geometry and let $Y\subset X$ be a subspace. If $X$ satisfies the twisted coarse Baum--Connes conjecture with respect to any stable coarse algebras, then so does $Y$. 
\end{theorem}

Another aim of this paper is to apply our new framework to study the higher index theory for relatively hyperbolic groups.
The concept of relative hyperbolicity has been extensively studied over the past two decades. For a group $G$ that is hyperbolic relative to a finite collection of peripheral subgroups $\{H_1,\ldots,H_N\}$, D.~Osin \cite{Osin2005AsymptoticDimensionofRelativelyhyperbolicGroups} proved that $G$ has finite asymptotic dimension if and only if each subgroup $H_i$ does. N.~Ozawa \cite{Ozawa2006BoundaryAmenabilityofRelative} established that $G$ has Yu's Property A if and only if each $H_i$ does. Furthermore, M.~Dadarlat and E.~Guentner \cite{DadarlatGuentner2007Uniform} showed that $G$ admits a coarse embedding into Hilbert space if and only if each $H_i$ does.

A natural question is whether the coarse Baum--Connes conjecture holds for the relative hyperbolic group when each subgroup $H_i$ satisfies the conjecture. Fukaya and Oguni \cite{FukayaOguni2012TheCoarseBaumConnesconjectureforrelatively} proved that if $G$ is a finitely generated group that is hyperbolic relative to a finite family of subgroups $\{H_1,\cdots ,H_N\}$ such that
\begin{enumerate}[(1)]
    \item each $H_i$ satisfies the coarse Baum--Connes conjecture, and
    \item each $H_i$ admits a finite $H_i$-simplicial
        complex $\underline{E}H_i$ which is a universal space for proper $H_i$-actions,
\end{enumerate}
then $G$ satisfies the coarse Baum--Connes conjecture. 

In our setting of twisted coarse Baum--Connes conjecture with respect to stable coarse algebras, we are able to remove the second assumption mentioned above. Namely, we obtain the following:
\begin{theorem}\label{main}
    Let $G$ be a finitely generated group which is hyperbolic relative to a finite family of subgroups $\{H_1,\cdots ,H_N\}$. 
    Then $G$ satisfies the twisted coarse Baum--Connes conjecture with respect to any stable coarse algebras if and only if each $H_i$ does.
\end{theorem}

The twisted coarse Baum--Connes conjecture with respect to any stable coarse algebras implies the usual coarse Baum--Connes conjecture. As an application, we can verify the coarse Baum--Connes conjecture for a larger class of metric spaces:

\begin{corollary}
  Let $G$ be a finitely generated group which is hyperbolic relative to a finite collection of subgroups ${H_1, H_2, \dots, H_n}$. If each $H_i$ is an extension of groups admitting a coarse embedding into Hilbert space, then $G$ satisfies the coarse Baum–Connes conjecture.  
\end{corollary}

Notably, the subgroups $H_i$ need not admit a coarse embedding into Hilbert space. This occurs, for example, in extensions constructed by G. Arzhantseva and R. Tessera in \cite{Arzhantseva-Tessera:Extension-Coarse-Embedding}. Such spaces provide a central motivation for this paper. Our results show that groups hyperbolic relative to such subgroups still satisfy the coarse Baum–Connes conjecture.

This paper is organized as follows. In Section 2, we recall basic concepts of coarse algebras and the twisted coarse Baum--Connes conjecture. Then prove permanence properties of the twisted coarse Baum--Connes conjecture with respect to any stable coarse algebras. In Section 3, we recall controlled $K$-theory and computational techniques. In Section 4, we clarify what we need to prove to complete the proof of Theorem~\ref{main} using the language of decomposition complexity. In Section 5, we complete the proof of Theorem~\ref{main}.

\section{The twisted coarse Baum--Connes conjecture with respect to stable coarse algebras}
In this section, we review some background on the coarse Baum–Connes conjecture together with its variants. More precisely, we will focus on a version of twisted coarse Baum-Connes conjecture with respect to certain stable coarse algebras. This is a new framework to study the coarse geometry of metric spaces with bounded geometry. Compared with the usual coarse Baum-Connes conjecture, the new framework captures more information of the large scale geometry of the considered space.

\subsection{Twisted Roe algebras}
 
In this subsection, we recall the definition of coarse algebras and use them to define the twisted Roe algebra for a metric space. The concept of coarse algebra was introduced by the first-named author and L.~Guo in \cite{DengGuo2024TwistedRoeAlgebras}, and it was used to study the coarse Baum--Connes conjecture for a coarse fibration whose base space and fiber admit a coarse embedding into Hilbert space. We introduce the notion of stable coarse algebras and prove that the corresponding twisted Roe algebras are invariant under coarse equivalence.

Let $X$ be a discrete metric space, and $\mathcal{A}$ a $C^*$-algebra. Let $\mathcal{K}$ be the algebra of all compact operators on an infinite-dimensional, separable Hilbert space $\mathcal{H}$ and $\mathcal{A}\otimes\mathcal{K}$ the tensor product. Denoted by $\ell^{\infty}(X, \mathcal{A}\otimes \mathcal{K})$ the $C^*$-algebra of all bounded, $(\mathcal{A}\otimes \mathcal{K})$-valued functions on $X$. We fix a $*$-isomorphism $\psi:\mathcal{K}\otimes \mathcal{K}\to \mathcal{K}$, and also the induced isomorphism ${\rm id}_{\mathcal{A}}\otimes \psi: \mathcal{A}\otimes \mathcal{K}\otimes \mathcal{K}\to \mathcal{A}\otimes \mathcal{K}$ is still denoted by $\psi$.

Let $D, R\subset X$ be subsets of $X$. A bijective map $v:D \to R$ is said to be a partial translation if we have 
$$
\sup_{x\in D}d(x,v(x))<\infty.
$$
For every function $f\in \ell^{\infty}(X,\mathcal{A}\otimes \mathcal{K})$ and any partial translation $v: D \to R$, we define
\[
v^*f(x)=
\begin{cases}
  f(v(x)),& \mbox{~if~} x\in D\\
  0,      & \mbox{otherwise.}
\end{cases}
\]

Now, we are ready to define coarse algebras and stable coarse algebras.
\begin{definition}\label{stable coarse algebra}
Let $X$ be a discrete metric space and $\mathcal{A}$ a $C^*$-algebras.  
    We say that a $C^*$-subalgebra $\Gamma(X,\mathcal{A}\otimes\mathcal{K})$ of $\ell^{\infty}(X,\mathcal{A}\otimes \mathcal{K})$ is a  coarse $X$-algebra if it satisfies the following conditions:
    \begin{enumerate}[(1)]
        \item For each $x\in X$, we have that $$\overline{\{f(x):f\in \Gamma(X, \mathcal{A}\otimes \mathcal{K})\}}=\mathcal{A}\otimes \mathcal{K};$$
        \item For each $f\in \Gamma(X, \mathcal{A}\otimes \mathcal{K})$ and each partial translation $v: D \to R$, we have that $v^*f\in  \Gamma(X, \mathcal{A}\otimes \mathcal{K})$.
    \end{enumerate}
In addition, a coarse $X$-algebra $\Gamma(X,\mathcal{A}\otimes\mathcal{K})$ is said to be stable if it satisfies the following conditions:         
    \begin{enumerate}
        \item[(3)] For any two bounded sequences $a,b\in \ell^{\infty}(X,B(\mathcal{H}))$, and $f\in \ell^{\infty}(X;\mathcal{A}\otimes \mathcal{K})$, we define a new function $f'\in \ell^{\infty}(X;\mathcal{A}\otimes \mathcal{K})$ by $$f'(x)=\left(\text{id}_{\mathcal{A}}\otimes a(x)\right)f(x)\left(\text{id}_{\mathcal{A}}\otimes b(x)\right).$$ Then $f\in \Gamma(X,\mathcal{A}\otimes\mathcal{K})$ implies $f'\in \Gamma(X,\mathcal{A}\otimes\mathcal{K})$ for any two bounded sequences $a,b\in \ell^{\infty}(X,B(\mathcal{H}))$;

        \item[(4)] For every $f\in \Gamma(X,\mathcal{A}\otimes \mathcal{K})$ and $k\in \mathcal{K}$, the function $x\mapsto \psi(f(x)\otimes k)$ is an element in $\Gamma(X,\mathcal{A}\otimes \mathcal{K})$.
    \end{enumerate}
\end{definition}
We remark here that Conditions $(3)$ and $(4)$ are used to guarantee the coarse invariance of twisted Roe algebras.

\begin{example}\label{example coarse algebra} Let $X$ be a discrete metric space and $\mathcal{A}$ a $C^*$-algebra.
\begin{enumerate}[(1)]
    \item The algebra $\ell^{\infty}(X,\mathcal{A}\otimes \mathcal{K})$ is a stable coarse $X$-algebra;
    \item The $C^*$-algebra $C_0(X, \mathcal{A}\otimes \mathcal{K})$ consisting of all bounded, $(\mathcal{A}\otimes \mathcal{K})$-valued functions vanishing at infinity is a stable coarse $X$-algebra.

    \item For a subspace $Y\subset X$, define  $\Gamma(X,\mathcal{A}\otimes\mathcal{K})$ to be the closure of
    \begin{align*}
        \Theta[X,\mathcal{A}\otimes\mathcal{K}]:=\{f\in \ell^{\infty}(X, \mathcal{A}\otimes \mathcal{K}):\exists r>0 \text{ s.t. } \supp (f) \subset N_r(Y)\},
    \end{align*}
    where $N_r(Y):=\{x\in X:d(x,Y)\leq r\}$ is the $r$-neighborhood of $Y$.
    Then $\Gamma(X,\mathcal{A}\otimes\mathcal{K})$ is a stable coarse $X$-algebra.
\end{enumerate}
\end{example}

Let $M$ be a proper metric space in the sense that every closed bounded subset is compact. We fix a net $X \subset M$ in $M$. Recall that a subset $X$ is called a net if there exists a positive number $C,D>0$ such that any pair of distinct points in $X$ has distance at least $C$, and $N_D(X)=M$, where $N_D(X)=\{m\in M: d(m,X)\leq D\}$. We recall the definition of bounded geometry.

\begin{definition}
    A (not necessarily discrete) metric space $M$ is said to have  bounded geometry if there exists $\delta>0$ such that for every $R>0$, there exists $K(R)\geq 0$ such that the number of points in a ball $B(x,R)$ with pairwise distance greater that $\delta$ does not exceed $K(R)$.
\end{definition}

From the definition of bounded geometry, it is easy to show the following Lemma (c.f. \cite{WillettYuHigherindextheoryBook}~Appendix A.1.10).
\begin{lemma}\label{borel cover}
If $M$ is a metric space with bounded geometry, then there exist a net $X\subset M$ and a Borel covering $\{B_x\}_{x\in X}$ for $M$ such that 
\begin{enumerate}[(1)]
    \item for each $x$, $B_x$ contains a nonempty open subset of $M$,
    \item for any $x,y \in X$ with $x\neq y$, we have $B_x\cap B_y=\emptyset$,
    \item for each $x\in X$, we have $x\in B_x$,
    \item $\sup_{x\in X} {\rm diam}(B_x)<\infty$.
\end{enumerate}
\end{lemma}

Let $\mathcal{A}$ be a $C^*$-algebra. We first recall the definition of Roe algebras with coefficients in $\mathcal{A}$.
Fix a countable dense subset $M_0\subset M$ and a separable, infinite-dimensional Hilbert space $\mathcal{H}$. We denote the standard Hilbert $\mathcal{A}$-module by $\mathcal{H}_{\mathcal{A}}:=\mathcal{A}\otimes \mathcal{H}$. We consider a Hilbert $\mathcal{A}$-module
$\ell^2(M_0, \mathcal{H}_{\mathcal{A}})$
with the following structure:
$$
\left\langle\sum_{m\in M_0} a_m, \sum_{m\in M_0} b_x\right\rangle=\sum_{m\in M_0}\left\langle a_m,b_m\right \rangle_{\mathcal{H}_{\mathcal{A}}},
$$
$$
\left(\sum_{m\in M_0} a_m\right)\cdot a=\sum_{m\in M_0} (a_ma)
$$
for all $\sum_{m\in M_0} a_m, \sum_{m\in M_0} b_m\in \ell^2(M_0, \mathcal{H}_{\mathcal{A}})$ and $a\in \mathcal{A}$. Let $B(\ell^2(M_0, \mathcal{H}_{\mathcal{A}}))$ be the algebra of all adjointable $\mathcal{A}$-linear operators on $\ell^2(M_0, \mathcal{H}_{\mathcal{A}})$.

It is obvious that every bounded $\mathcal{A}$-linear operator $T \in B(\ell^2(M_0, \mathcal{H}_{\mathcal{A}}))$ can be expressed as an $(M_0 \times M_0)$-matrix $\left(T_{xy}\right)_{x,y\in M_0}$, where $T_{xy}\in \mathcal{A}\otimes B(\mathcal{H})$ for all $x, y\in M_0$. 

Now we are ready to recall some basic concepts in coarse geometry.

\begin{definition}
    Let $T=\left(T_{xy}\right)_{x,y\in M_0}$ be an adjointable operator on $\ell^2(M_0,\mathcal{H}_{\mathcal{A}}) $.
    \begin{enumerate}[(1)]
        \item The propagation of $T$ is defined to be 
        $$
        {\rm Propagation}(T):=\sup\{d(x,y):T_{x,y}\neq 0\}\in [0,\infty].
        $$
        We say that $T$ has finite propagation if ${\rm Propagation}(T)<\infty$.
        \item The operator $T$ is said to be locally compact, if for every bounded subset $B\subset M$, the operator $\left(T_{xy}\right)_{x,y\in B\cap M_0}$ is a compact operator on $\ell^2(B\cap M_0, \mathcal{H}_{\mathcal{A}})$.
    \end{enumerate}
\end{definition}
 The concept of Roe algebra was introduced by J. Roe. 
\begin{definition}
   Let $M$ be a metric space with bounded geometry, $M_0\subset M$ a countable dense subset, and $\mathcal{A}$ a $C^*$-algebra. The Roe algebra of $M$ with coefficients in $\mathcal{A}$ is the $C^*$-subalgebra $C^*(M,\mathcal{A})$ of $B(\ell^2(M_0, \mathcal{A}\otimes \mathcal{H}))$ generated by all locally compact operators with finite propagation. In particular, when $\mathcal{A}=\mathbb{C}$, $C^*(M,\mathbb{C})$ is denoted by $C^*(M)$.
\end{definition}

Next, we recall the definition of twisted Roe algebras. This notion was first introduced by the first-named author together with L.~Guo in \cite{DengGuo2024TwistedRoeAlgebras}, where the coefficient algebras were taken to be coarse algebras. In the present paper, we instead use stable coarse algebras to define twisted Roe algebras. With this modification, the resulting twisted Roe algebras are invariant under coarse equivalences in a certain sense (Theorem~\ref{coarse equivalence}).

We fix a net $X$ of $M$ and a Borel covering $\{B_x\}_{x\in X}$ satisfying the four conditions in Lemma~\ref{borel cover}.
For every $x\in X$, fix a unitary 
\begin{align}\label{identification of Hilbert spaces}
    U_x:\ell^2(B_x\cap M_0)\otimes\mathcal{H}\to \mathcal{H}
\end{align}
between Hilbert spaces. This induces an isomorphism on the Hilbert $A$-modules $U_x\otimes \text{id}_{\mathcal{A}}:\ell^2(B_x\cap M_0)\otimes\mathcal{H}_{\mathcal{A}}\to \mathcal{H}_{\mathcal{A}}$.

For every $T=\left(T_{xy}\right)_{x,y\in M_0}\in B(\ell^2(M_0, \mathcal{H}_{\mathcal{A}}))$ and every partial translation $v:D \to R$ on $X$, we define a bounded function
$$
T^v:X\to \mathcal{A}\otimes \mathcal{K}
$$
by 
\begin{align*}
    T^v(x):=\left(\text{id}_{\mathcal{A}}\otimes U_x\right) \chi_{B_x}T\chi_{B_{v(x)}} \left(\text{id}_{\mathcal{A}}\otimes U_{v(x)}\right)\in \mathcal{A}\otimes \mathcal{K}.
\end{align*}
Note that for a stable coarse $X$-algebra $\Gamma(X,\mathcal{A}\otimes \mathcal{K})$ and $T\in B(\ell^2(M_0, \mathcal{H}_{\mathcal{A}}))$, the condition that $T^v\in \Gamma(X,\mathcal{A}\otimes \mathcal{K})$ is independent of the choice of unitaries $\{U_x\}$ by the third condition in Definition~\ref{stable coarse algebra}. For brevity, we  omit $\text{id}_{\mathcal{A}}\otimes U_x$. 

\begin{definition}\label{definition of twisted roe algebra}
    Let $M$ be a metric space with bounded geometry, $X\subset M$ a net in $M$, and $\left\{B_x\right\}_{x\in X}$ a Borel cover satisfying the four conditions in Lemma~\ref{borel cover}. Let $\mathcal{A}$ be a $C^*$-algebra, and $\Gamma(X, \mathcal{A}\otimes \mathcal{K})$ a coarse $X$-algebra. 
    \begin{enumerate}[(1)]
    \item The algebraic $\Gamma(X, A\otimes \mathcal{K})$-twisted Roe algebra of $M$, denoted by $\mathbb{C}[M;\Gamma(X,\mathcal{A}\otimes \mathcal{K})]$, is the $*$-subalgebra of the algebraic Roe algebra $\mathbb{C}[M,\mathcal{A}\otimes \mathcal{K}]$ consisting of all operator $T$ satisfying that $T^v\in \Gamma(X, \mathcal{A}\otimes \mathcal{K})$ for all partial translation $v: D \to R$ on $X$.

    \item The $\Gamma(X,\mathcal{A}\otimes \mathcal{K})$-twisted Roe algebras $C^*(M;\Gamma(X,\mathcal{A}\otimes \mathcal{K}))$ is the completion of the $*$-algebra $\mathbb{C}[M;\Gamma(X,\mathcal{A}\otimes \mathcal{K})]$ under the operator norm on $\ell^2(M_0, \mathcal{H}_{\mathcal{A}})$.
    \end{enumerate}
\end{definition}

It will be shown in Lemma~\ref{independence of nets} that the definition of the $\Gamma(X,\mathcal{A}\otimes \mathcal{K})$-twisted Roe algebras are independent of the choice of a Borel cover $\{B_x\}$ satisfying the four conditions in Lemma~\ref{borel cover}.

\begin{example} Let $X$ be a metric space with bounded geometry and $\mathcal{A}$ a $C^*$-algebra. 
\begin{enumerate}[(1)]
    \item For the stable coarse $X$-algebra $\ell^{\infty}(X,\mathcal{A}\otimes\mathcal{K})$, we have that
    $$
    C^*(X, \ell^{\infty}(X, \mathcal{A}\otimes\mathcal{K}))=C^*(X,\mathcal{A}).
    $$
    \item For the stable coarse $X$-algebra $C_0(X, \mathcal{A}\otimes \mathcal{K})$, we have that
    $$
    C^*(X, C_0(X, \mathcal{K}))\cong \mathcal{A}\otimes\mathcal{K}.
    $$

    \item For a subspace $Y\subset X$ and the stable coarse $X$-algebra $\Gamma(X,\mathcal{A}\otimes\mathcal{K})$ defined in Example \ref{example coarse algebra} (3), we have that
    \begin{align*}
        C^*(X;\Gamma(X,\mathcal{A}\otimes\mathcal{K}))\cong C^*(Y,\mathcal{A}).
    \end{align*}
\end{enumerate}
    
\end{example}

In \cite{DengGuo2024TwistedRoeAlgebras}, the first-named author and L. Guo proved that the twisted coarse Baum--Connes conjecture satisfies the coarse imprimitivity. This can be used to compute the $K$-theory of twisted Roe algebras.

We now prove the coarse invariance. In Proposition~\ref{easiest coarse equivalence}, we assume that the coarse equivalence $f:M \to N$ restricts to a bijection between the fixed nets of $X$ and $Y$, and that it preserves the chosen Borel coverings. In this case, for a given a coarse $Y$-algebra $\Gamma(Y,\mathcal{A}\otimes \mathcal{K})$, we consider the coarse $X$-algebra
$$f^*(\Gamma(Y,\mathcal{A}\otimes \mathcal{K}))=\{
\varphi\circ f: \varphi \in  \Gamma(Y,\mathcal{A}\otimes \mathcal{K})
\}.$$
Then in this case, the proof is the same as that of coarse invariance Roe algebras, and it is not necessary to require that the coarse algebras are stable. In Theorem~\ref{coarse equivalence}, we prove a more general coarse invariance.

\begin{proposition}\label{easiest coarse equivalence}    Let $f:M\to N$ be a coarse equivalence between metric spaces with bounded geometry.  Fix a net $X$ in $M$ such that $f$ is one-to-one on $X$ and $Y:=f(X)$ is a net of $Y$, and fix Borel coverings $\{B_x\}_{x\in X}$ of $M$ and $\{C_y\}_{y\in Y}$ of $N$ satisfying the four conditions in Lemma~\ref{borel cover} and $f(B_x)\subset C_{f(x)}$ for all $x\in X$. For any coarse $Y$-algebra $\Gamma(Y,\mathcal{A}\otimes \mathcal{K})$ and the coarse $X$-algebra defined by $\Gamma(X,\mathcal{A}\otimes \mathcal{K}):=f^*(\Gamma(Y,\mathcal{A}\otimes \mathcal{K}))$, we have an isomorphism
    \begin{align*}
        C^*(M;\Gamma(X,\mathcal{A}\otimes \mathcal{K}))\cong C^*(N;\Gamma(Y,\mathcal{A}\otimes \mathcal{K})).
    \end{align*}
\end{proposition}

\begin{proof}
Let $U_x:\ell^2(B_x\cap M_0,\mathcal{H})\to \mathcal{H}$ and $V_y:\ell^2(C_y\cap N_0,\mathcal{H})\to \mathcal{H}$ be the fixed unitaries in \eqref{identification of Hilbert spaces} between Hilbert spaces for all $x\in X$ and $y\in Y$.
    For each $x\in X$, we denote $W_x:=V_{f(x)}^{*}U_x:\ell^2(B_x\cap M_0,\mathcal{H})\to \ell^2(C_{f(x)}\cap N_0,\mathcal{H})$ and define 
    \begin{align*}
        W:=\bigoplus_x (W_x\otimes \text{id}_{\mathcal{A}}):\ell^2( M_0,\mathcal{H}_{\mathcal{A}})\to \ell^2(N_0,\mathcal{H}_{\mathcal{A}}).
    \end{align*}
    One can check that for any $T\in B(\ell^2(M_0,\mathcal{H}_{\mathcal{A}}))$ and $x,y\in X$, we have
    \begin{align*}
        \left(V_{f(x)}\otimes \text{id}_{\mathcal{A}}\right)\chi_{C_{f(x)}}\left(WTW^*\right)&\chi_{C_{f(y)}}\left(V_{f(y)}^*\otimes \text{id}_{\mathcal{A}}\right)\\
        &=\left(V_{f(x)}\otimes \text{id}_{\mathcal{A}}\right)\left(\chi_{B_x}T\chi_{B_y}\right)\left(V_{f(y)}^*\otimes \text{id}_{\mathcal{A}}\right).
    \end{align*}
    Then it is routine to check $WC^*(M;\Gamma(X,\mathcal{A}\otimes \mathcal{K}))W^*=C^*(N;\Gamma(Y,\mathcal{A}\otimes \mathcal{K}))$.
\end{proof}

Next, we show that twisted Roe algebras with stable coarse algebras are independent of the choice of nets and Borel coverings.

We begin by explaining how a coarse equivalence \(f : X \to Y\) induces, via pullback, a stable coarse \(X\)-algebra \(\Gamma(X, \mathcal{A} \otimes \mathcal{K})\) from a stable coarse \(Y\)-algebra \(\Gamma(Y, \mathcal{A} \otimes \mathcal{K})\).
Then, we use this to formulate the independence of the choice of nets and coarse invariance of twisted Roe algebras. Here it is crucial to assume that coarse algebras are stable (Definition~\ref{stable coarse algebra}). Let $X$ and $Y$ be discrete metric spaces with bounded geometry, $f:X\to Y$ a coarse equivalent map, and $\Gamma(Y,\mathcal{A}\otimes \mathcal{K})$ a stable coarse $Y$-algebra. For $\varphi\in \ell^{\infty}(X,\mathcal{A}\otimes \mathcal{K})$ and any subset $\Tilde{X}$ of representatives $\{x_y:x_y\in f^{-1}(y),~ y\in f(X)\}$ from each preimage $f^{-1}(y)$, define $\varphi_{\Tilde{X}}\in \ell^{\infty}(Y,\mathcal{A}\otimes \mathcal{K})$ by
\begin{align*}
    \varphi_{\Tilde{X}}(y)=\left\{
    \begin{array}{cc}
        \varphi(x), & y=f(x), x\in \Tilde{X},   \\
         0, & \text{otherwise.}
    \end{array}
    \right.
\end{align*}
Define 
\begin{align}\label{twist for coarse equi}
\Gamma(X,\mathcal{A}\otimes \mathcal{K})
  := \left\{\, \varphi\in \ell^{\infty}(X,\mathcal{A}\otimes \mathcal{K}) :
     \begin{array}{l}
       \varphi_{\Tilde{X}}\in \Gamma(Y,\mathcal{A}\otimes \mathcal{K}) 
       \text{ for all subset }\\  \Tilde{X}  \text{ of representatives of preimages}
     \end{array}
   \right\}.
\end{align}

\begin{lemma}
    The set $\Gamma(X,\mathcal{A}\otimes \mathcal{K})$ is a stable coarse $X$-algebra.
\end{lemma}

\begin{proof}
    It is obvious that $\Gamma(X,\mathcal{A}\otimes \mathcal{K})$ is a $C^*$-subalgebra of $\ell^{\infty}(X,\mathcal{A}\otimes \mathcal{K})$. For any partial translation $v$ on $X$ and any $\varphi\in \Gamma(X,\mathcal{A}\otimes \mathcal{K})$, we show that $v^*\varphi \in \Gamma(X,\mathcal{A}\otimes \mathcal{K})$.
    For any subset $\Tilde{Y}\subset f(X)$, we fix a set of representatives 
    \begin{align*}
        \Tilde{X}:=\{x_y:x_y\in f^{-1}(y), y\in \Tilde{Y}\}
    \end{align*}
    from each preimage $f^{-1}(y)$ for every $y\in \Tilde{Y}$. Then for any $y\in \Tilde{Y}$ such that $x_y\in \dom(v)$,
    \begin{align*}
        (v^*\varphi)_{\Tilde{X}}(y)=(v^*\varphi)(x_y)=\varphi(v(x_y))
    \end{align*}
    and $(v^*\varphi)_{\Tilde{X}}(y)=0$ if $y\notin \Tilde{Y}$ or $x_y\notin \dom(v)$.
    
    Decompose $\{y\in \Tilde{Y}:x_y\in \dom(v)\}=\bigsqcup_{j=1}^N A_j$ such that for any $j=1,2,\cdots, N$ and for any distinct points $y,y'\in A_j$, we have that $f(v(x_y))\neq f(v(x_{y'}))$. Then we can define a partial translations $v_j$ on $Y$ such that $\dom(v_j)= A_j$ and $v_j(y)=f(v(x_y))$. Define a set 
    \begin{align*}
        \Tilde{X}_j:=\{v(x_y):y\in A_j\},
    \end{align*}
  which contains at most one point from each preimage $f^{-1}(y)$. Then for any $y\in A_j$,
  \begin{align*}
      v_j^*(\varphi_{\Tilde{X}_j})(y)=\varphi_{\Tilde{X}_j}(v_j(y))=\varphi_{\Tilde{X}_j}(f(v(x_y)))=\varphi(v(x_y))
  \end{align*}
  and $v_j^*(\varphi_{\Tilde{X}_j})(y)=0$ if $y\notin A_j$. Therefore $v^*\varphi=\sum_{j=1}^N v_j^*(\varphi_{\Tilde{X}_j})\in \Gamma(X,\mathcal{A}\otimes \mathcal{K})$. It is easy to see $\Gamma(X,\mathcal{A}\otimes \mathcal{K})$ is stable if $\Gamma(Y,\mathcal{A}\otimes \mathcal{K})$ is.
\end{proof}

By the same technique of decomposition as above, we have the following. 
\begin{lemma}\label{well-defined up to close}
Let $X$ and $Y$ be metric spaces with bounded geometry, and $\Gamma(Y, \mathcal{A}\otimes \mathcal{K})$ a stable coarse $Y$-algebra.  If $f,g:X\to Y$ are coarse equivariant maps which are close in the sense that $$\sup_{x\in X} d(f(x),g(x))<\infty,$$ 
then $f$ and $g$ induce the same coarse $X$-algebra $\Gamma(X,\mathcal{A}\otimes \mathcal{K})$ by the formula \eqref{twist for coarse equi}.
\end{lemma}

\begin{remark}\label{functoriarity}
    For two coarse equivalences $f:X \to Y$ and $g:Y\to Z$ between discrete metric spaces with bounded geometry. Assume $\Gamma(Z,\mathcal{A}\otimes\mathcal{K})$ is a coarse $Z$-algebra and $\Gamma(Y,\mathcal{A}\otimes\mathcal{K})$ is the coarse $Y$-algebra induced by $g$ via \eqref{twist for coarse equi}. Denote by $\Gamma_1(X,\mathcal{A}\otimes\mathcal{K})$ the coarse $X$-algebra induced by $f$ from $\Gamma(Y,\mathcal{A}\otimes\mathcal{K})$ and denote by $\Gamma_2(X,\mathcal{A}\otimes\mathcal{K})$ the coarse $X$-algebra induced by $g\circ f$ from $\Gamma(Z,\mathcal{A}\otimes\mathcal{K})$. Then it is easy to see that 
    \begin{align*}
        \Gamma_1(X,\mathcal{A}\otimes\mathcal{K})=\Gamma_2(X,\mathcal{A}\otimes\mathcal{K}).
    \end{align*}
\end{remark}

Now, we are ready to show that the definition of twisted Roe algebras is independent of the choice of nets and Borel coverings.
\begin{lemma}\label{independence of nets}
    Let $M$ be a metric space with bounded geometry, and let $X$ and $Y$ be nets of $M$ with associated Borel covers $\{B_x\}_{x\in X}$ and $\{C_y\}_{y\in Y}$ of $M$ satisfying the four conditions in Lemma~\ref{borel cover}. Suppose that $\Gamma(Y,\mathcal{A}\otimes \mathcal{K})$ is a stable coarse $Y$-algebra and $\Gamma(X,\mathcal{A}\otimes \mathcal{K})$ is a coarse $X$-algebra induced by any coarse equivalence between $X$ and $Y$, which is close to the identity map on $M$, by the formula \eqref{twist for coarse equi}. Denote by $C^*(M;\Gamma(X,\mathcal{A}\otimes \mathcal{K}))$ (resp. $C^*(M;\Gamma(Y,\mathcal{A}\otimes \mathcal{K}))$) the $\Gamma(X,\mathcal{A}\otimes \mathcal{K})$-twisted (resp. $\Gamma(Y,\mathcal{A}\otimes \mathcal{K})$-twisted) Roe algebra associated to the Borel cover $\{B_x\}_{x\in X}$ (resp. $\{C_y\}_{y\in Y}$). Then we have an equality 
    \begin{align*}
        C^*(M;\Gamma(X,\mathcal{A}\otimes \mathcal{K}))=C^*(M;\Gamma(Y,\mathcal{A}\otimes \mathcal{K}))
    \end{align*}
    in $B(\ell^2(M_0,\mathcal{H}_{\mathcal{A}}))$.
\end{lemma}

\begin{proof}
We fix a coarse equivalence $f:X\to Y$, which is close to the identity on $M$.
    By Remark~\ref{functoriarity}, it suffices to show 
    \begin{align*}
        \mathbb{C}[M;\Gamma(Y,\mathcal{A}\otimes \mathcal{K})]\subset \mathbb{C}[M;\Gamma(X,\mathcal{A}\otimes \mathcal{K})].
    \end{align*}
    Let $T\in \mathbb{C}[M;\Gamma(Y,\mathcal{A}\otimes \mathcal{K})]$ and $v$ a partial translation in $X$. We show that the function $\dom(v)\ni x\mapsto \chi_{B_{x}}T\chi_{B_{v(x)}}\in \mathcal{A}\otimes \mathcal{K}$ is in $\Gamma(X,\mathcal{A}\otimes\mathcal{K})$.

    For each $x\in X$, define 
    $$D_x:=\bigcup \{C_y:C_y\cap B_x\neq \phi\}\supseteq B_x.$$ By the third condition in Definition~\ref{stable coarse algebra}, it suffices to show that the function 
    $$\dom(v)\ni x\mapsto \chi_{D_{x}}T\chi_{D_{v(x)}}\in \mathcal{A}\otimes \mathcal{K}$$
    is in $\Gamma(X,\mathcal{A}\otimes\mathcal{K})$.
    We label $\{C_y:C_y\cap B_x\neq \phi\}=\{C_{y_1^x},C_{y_2^x},\cdots, C_{y_{n_x}^x}\}$, where $n_x$ is uniformly bounded by some $n$. We regard the operator $\chi_{D_{x}}T\chi_{D_{v(x)}}$ is an $n_x$-by-$n_{v(x)}$ matrix with entry in $\mathcal{A}\otimes \mathcal{K}$. By the fourth condition of Definition~\ref{stable coarse algebra}, it suffices to show that the function $\varphi:\dom(v)\ni x\mapsto \chi_{C_{y_i^x}}T\chi_{C_{y_j^{v(x)}}}\in \mathcal{A}\otimes \mathcal{K}$ is in $\Gamma(X,\mathcal{A}\otimes\mathcal{K})$ for any $i$ and $j$. For any subset $\Tilde{Y}\subset f(X)$, we fix a set  
    $$\Tilde{X}=\{x_y:x_y\in f^{-1}(y), y\in \Tilde{Y}\subset f(X)\}$$ of representatives of preimages and show that $\varphi_{\Tilde{X}}\in \Gamma(Y,\mathcal{A}\otimes \mathcal{K})$. We may assume $\Tilde{X}\subset \dom(v)$. We decompose $\Tilde{Y}=\sqcup_{k=1}^N A_k$ such that:
    \begin{itemize}
   
        \item for any distinct $y_1,y_2\in A_k$, we have $y_i^{x_{y_1}}\neq y_i^{x_{y_2}}$, and

        \item for any distinct $y_1,y_2\in A_k$, we have $y_j^{v(x_{y_1})}\neq y_j^{v(x_{y_2})}$.
    \end{itemize}
    Then we can define partial translations $u_k$ and $w_k$ such that
    \begin{itemize}
        \item $\dom(u_k)=\dom(w_k)=A_k$, and

        \item $u_k(y)=y_i^{x_{y}}$ and $w_k(y)=y_j^{v(x_{y})}$.
   \end{itemize}
    Then for $y\in A_k$, we have that 
    $$\varphi_{\Tilde{X}}(y)=\varphi(x_y)=\chi_{C_{{y_i}^{x_y}}}T\chi_{C_{y_j}^{v(x_y)}}= \chi_{C_{u_k(y)}}T\chi_{C_{w_k(y)}}.$$ 
    Therefore, $\varphi_{\Tilde{X}}=\sum_k u_k^*\left( T^{w_k\circ u_k^{-1}}\right)\in \Gamma(Y,\mathcal{A}\otimes\mathcal{K})$. 
\end{proof}

Combining Proposition~\ref{easiest coarse equivalence} and Lemma~\ref{independence of nets}, we can prove the following coarse invariance of twisted Roe algebras.
\begin{theorem}\label{coarse equivalence}
   Let $M$ and $N$ be metric spaces with bounded geometry, and let $X \subset M$ and $Y \subset N$ be nets of $M$ and $N$, respectively. Let $f : M \to N$ be a coarse equivalence, and fix Borel covers $\{B_x\}_{x \in X}$ and $\{C_y\}_{y \in Y}$ satisfying four conditions in Lemma~\ref{borel cover}, together with a coarse equivalence $g : X \to Y$ that is close to $f|_X$.

For a given coarse $Y$-algebra $\Gamma(Y, \mathcal{A} \otimes \mathcal{K})$, let $\Gamma(X, \mathcal{A} \otimes \mathcal{K})$ be the coarse $X$-algebra defined via $g$ as in \eqref{coarse alg of subspace}. Then there is an isomorphism
\[
    C^*(M; \Gamma(X, \mathcal{A} \otimes \mathcal{K})) \;\cong\; C^*(N; \Gamma(Y, \mathcal{A} \otimes \mathcal{K})).
\]
\end{theorem}

\begin{proof}
    We can take nets $Z\subset M$ and $W\subset N$ such that $f:Z\to W$ is bijective. We can take Borel covers $\{B'_z\}_{z\in Z}$ and $\{C'_w\}_{w\in W}$ such that $f(B'_z)\subset C'_{f(z)}$. Let $\Gamma(W,\mathcal{A}\otimes\mathcal{K})$ be the stable coarse algebra induced by any coarse equivalence $Y\to W$ that is close to the inclusion $Y\hookrightarrow N$. By defining $\Gamma(Z,\mathcal{A}\otimes\mathcal{K}):=f^*\Gamma(W,\mathcal{A}\otimes\mathcal{K})$, by Proposition~\ref{easiest coarse equivalence} we have
    \begin{align*}
        C^*(M;\Gamma(Z,\mathcal{A}\otimes\mathcal{K}))\cong C^*(N;\Gamma(Y,\mathcal{A}\otimes\mathcal{K})).
    \end{align*}
    But by Lemma~\ref{independence of nets} and Remark~\ref{functoriarity}, we have
    \begin{align*}
      \quad \quad\quad \quad\quad \quad \quad\quad \quad \quad\quad C^*(M;\Gamma(X,\mathcal{A}\otimes\mathcal{K}))= C^*(M;\Gamma(Z,\mathcal{A}\otimes\mathcal{K})).\quad \quad \quad \quad\quad \quad \quad\quad\quad\quad \quad\qedhere
    \end{align*}
\end{proof}

We will use the following theorem to prove the permanence of the twisted coarse Baum–Connes conjecture with respect to stable coarse algebras under passing to subspaces.

\begin{theorem}\label{twist for subsp.}
    Let $X$ be a metric space with bounded geometry and $Y\subset X$ a subsapce. For a given coarse $Y$-algebra $\Gamma(Y,\mathcal{A}\otimes \mathcal{K})$, there exists a coarse $X$-algebra $\Gamma(X,\mathcal{A}\otimes \mathcal{K})$ such that
    \begin{align*}
        C^*(X;\Gamma(X,\mathcal{A}\otimes \mathcal{K}))\cong C^*(Y;\Gamma(Y,\mathcal{A}\otimes \mathcal{K})).
    \end{align*}
\end{theorem}

\begin{proof}
    For $r>0$, define $N_r(Y):=\{x\in X:d(x,Y)\leq r\}$. Fix a coarse equivalence $q_r:N_r(Y)\to Y$, which fixes every point in $Y$. Then, this map induces a coarse $N_r(Y)$-algebra $\Gamma(N_r(Y),\mathcal{A}\otimes \mathcal{K})$ by \eqref{twist for coarse equi}, which is independent of the choice of $q_r$ by Lemma \ref{well-defined up to close}. Therefore for $s<r$, we have $\Gamma(N_s(Y),\mathcal{A}\otimes \mathcal{K})\subset\Gamma(N_r(Y),\mathcal{A}\otimes \mathcal{K})$ and the inductive limit 
    \begin{align}\label{coarse alg from subspace}
        \Gamma(X,\mathcal{A}\otimes \mathcal{K}):=\varinjlim \Gamma(N_r(Y),\mathcal{A}\otimes \mathcal{K})
    \end{align}
    is an well-defined coarse $X$-algebra. We claim 
    \begin{align*}
        C^*(X;\Gamma(X,\mathcal{A}\otimes \mathcal{K}))= \varinjlim C^*(N_r(Y);\Gamma(N_r(Y),\mathcal{A}\otimes \mathcal{K})) \subset B(\ell^2(X,\mathcal{H}_{\mathcal{A}})).
    \end{align*}
    We show that $\mathbb{C}[X;\Gamma(X,\mathcal{A}\otimes \mathcal{K})] \subset \varinjlim C^*(N_r(Y);\Gamma(N_r(Y),\mathcal{A}\otimes \mathcal{K}))$. Assume $T$ is in the left hand side with propagation less than $r_0$ for some $r_0$. We denote $E_{r_0}:=\{(x,x')\in X\times X:d(x,x')<r_0\}$ and decompose
    \begin{align*}
    E_{r_0}\subset  \bigsqcup_{j=1}^N (E(v_j)\sqcup E(v_j^{-1}))
    \end{align*}
    for some partial translations $v_1,\cdots ,v_N$ such that $\dom(v_j)\cap \ran (v_j)=\emptyset$ for all $j=1,\cdots ,N$, where 
    \begin{align*}
        E(v_j):=\{(x,v_j(x))\in X \times X :x\in \dom(v_j)\}. 
    \end{align*}
    For $r>0$ such that $T^{v_j}\in \Gamma(N_r(Y),\mathcal{A}\otimes\mathcal{K})$ for $j=1,\cdots, N$ (such $r$ exists by perturbing by taking arbitrary small perturbation of $T$, if necessary), we have $\supp (T)\subset N_r(Y)\times N_r(Y)$ and this means $T\in C^*(N_r(Y);\Gamma(N_r(Y),\mathcal{A}\otimes \mathcal{K}))$. The other inclusion is obvious. Therefore by Theorem~\ref{coarse equivalence}, we have 
    \begin{align*}
      \quad\quad\quad\quad  C^*(X;\Gamma(X,\mathcal{A}\otimes \mathcal{K}))= \varinjlim C^*(N_r(Y);\Gamma(N_r(Y),\mathcal{A}\otimes \mathcal{K})) \cong C^*(Y;\Gamma(Y,\mathcal{A}\otimes \mathcal{K})).\quad\quad\quad\quad \qedhere
    \end{align*}
\end{proof}

\subsection{The twisted coarse Baum--Connes conjecture}
In this subsection, we shall introduce the twisted coarse Baum--Connes conjecture for a metric space with respect to stable coarse algebras. 

To introduce the concept of twisted coarse $K$-homology, we shall first recall the definition of Rips complex. 

\begin{definition}
    Let $X$ be a metric space with bounded geometry and let $s>0$. The Rips complex at scale $s$, denoted by $P_s(X)$, is the simplicial complex whose vertices are $X$, and a subset $\{x_1,x_2,\dots, x_n\}$ spans a simplex if $d(x_i,x_j)\leq s$ for all $1\leq i,j\leq k$.
\end{definition}
Each Rips complex $P_d(X)$ can be equipped with the normalized spherical metric as follows. The simplex spanned by $\{x_1,x_2,\cdots,x_n\}$ is identified with the quadrant with diameter $1$:
$$S^{n-1}_+=\left\{(a_1,a_2,\cdots, a_n):\sqrt{\sum_{i=1}^na_i^2}=\frac{2}{\pi},~a_i\geq 0,\forall~1\leq i\leq n\right\}$$
via the map
$$
\sum_{j=1}^n a_jx_j\mapsto \frac{2}{\pi}\left(\frac{a_1}{\sum_{i=1}^n a_i^2},\cdots, \frac{a_n}{\sum_{i=1}^na_i^2} \right).
$$
Next, the Rips complex $P_s(X)$ is equipped with the path metric, and the distance between different connected components is defined to be infinity.  
For each $s>0$, we consider a countable dense subset $P_r(X)_0\subset P_r(X)$ such that $P_s(X)_0\subset P_{s'}(X)_0$ if $s< s'$. 

Let $M$ be a metric space with bounded geometry, $X\subset M$ a net and $\{B_x\}_{x\in X}$ a Borel cover of $M$ which satisfies the four conditions in Lemma~\ref{borel cover}. For a $C^*$-algebra $\mathcal{A}$, we consider a stable coarse $X$-algebra $\Gamma(X, \mathcal{A}\otimes\mathcal{K})$. Now we are ready to define the twisted localization algebras. 
\begin{definition}\label{localization}
    The $\Gamma(X,\mathcal{A}\otimes \mathcal{K})$-twisted algebraic localization algebra, which is denoted by $\mathbb{C}_L[M; \Gamma(X,\mathcal{A}\otimes \mathcal{K})]$, is defined to be the $*$-algebra of all uniformly bounded and uniformly continuous maps $f: [1,\infty)\to \mathbb{C}[M; \Gamma(X,\mathcal{A}\otimes\mathcal{K})]$ satisfying
    $${\rm Propagation}(f(t))\to 0,~\mbox{as}~t\to \infty.$$
    The $\Gamma(X,\mathcal{A}\otimes \mathcal{K})$-twisted localization algebra $C^*_L(M, \Gamma(X, \mathcal{A}\otimes\mathcal{K}))$ is the completion of the $*$-algebra $\mathbb{C}_L[M; \Gamma(X,\mathcal{A}\otimes\mathcal{K})]$ under the norm
    $$
    \|f\|=\sup_{t\in [1, \infty)}\|f(t)\|
    $$
    for each $f\in \mathbb{C}_L[P_r(X); \Gamma(X,\mathcal{A}\otimes\mathcal{K})]$.
\end{definition}

For each $s>0$, there is a natural evaluation-at-one map
\begin{align*}
  e: C^*_L(P_s(X), \Gamma(X,\mathcal{A}\otimes\mathcal{K})) &\to C^*(P_s(X), \Gamma(X,\mathcal{A}\otimes\mathcal{K}))\\ 
  f &\mapsto f(1).
\end{align*}
This is a $*$-homomorphism and it induces a homomorphism
\begin{align}\label{assembly}
   e_*:K_*(C^*_L(P_s(X), \Gamma(X,\mathcal{A}\otimes\mathcal{K}))) \to K_*(C^*(P_s(X), \Gamma(X,\mathcal{A}\otimes\mathcal{K}))). 
\end{align}
Since the right hand side of \eqref{assembly} is independent of $s$ by Theorem~\ref{coarse equivalence}, passing to the limit yields the twisted coarse assembly map 
$$\mu:\lim_{s\to \infty}K_*(C^*_L(P_s(X), \Gamma(X,\mathcal{A}\otimes\mathcal{K}))) \to K_*(C^*(X, \Gamma(X,\mathcal{A}\otimes\mathcal{K}))).$$

\begin{definition}[\textbf{The twisted coarse Baum-Connes conjecture with respect to any stable coarse algebras}]\label{twisted coarse baum--connes}
A discrete metric space $X$ with bounded geometry is said to satisfy the \emph{twisted coarse conjecture with respect to any stable coarse algebras} if the twisted coarse Baum--Connes assembly map
$$\mu:\lim_{s\to \infty}K_*(C^*_L(P_s(X), \Gamma(X,\mathcal{A}\otimes\mathcal{K}))) \to K_*(C^*(X, \Gamma(X,\mathcal{A}\otimes\mathcal{K}))),$$
is an isomorphism for any stable coarse algebra $\Gamma(X,\mathcal{A}\otimes\mathcal{K})$.
The kernel of the evaluation map $e$ is denoted by $C^*_{L,0}(P_r(X), \Gamma(X,\mathcal{A}\otimes\mathcal{K}))$ and is called the obstruction algebra. Then by the six-term exact sequence of $K$-theory of $C^*$-algebras, $\mu$ is isomorphic if and only if 
\begin{align*}
   \lim_{s\to \infty} K_*(C^*_{L,0}(P_s(X), \Gamma(X,\mathcal{A}\otimes\mathcal{K})))=0.
\end{align*}
\end{definition}

For a metric space with bonded geometry, if it admits a coarse embedding into Hilbert space, then it satisfies the twisted coarse Baum--Connes conjecture with respect to any stable coarse algebra by \cite[Theorem~3.7]{DengGuo2024TwistedRoeAlgebras}. In this paper, we shall study the twisted coarse Baum--Connes conjecture with respect to any coarse algebras for relative hyperbolic groups.




\subsection{Permanence properties of the twisted coarse Baum--Connes conjecture with respect to stable coarse algebras}
In this subsection we prove the permanence properties of the twisted coarse Baum--Connes conjecture with respect to any stable coarse algebras under coarse invariance and taking subspaces.

We first establish how a continuous and coarsely equivalent map $f:M \to N$ induces a map on the $K$-theory of their twisted localization algebras. This construction allows us to show the coarse invariance of the twisted coarse Baum–Connes conjecture. We proceed by recalling the necessary terminology (cf.~\cite[Definition 4.3.3]{WillettYuHigherindextheoryBook}).

\begin{definition}
    For a map $f:M\to N$ and $\varepsilon>0$, an isometry $V:\ell^2(M_0;\mathcal{H}) \to \ell^2(N_0;\mathcal{H})$ between Hilbert spaces is called an $\varepsilon$-cover of $f$ if
    \begin{align*}
        \supp(V)\subset \{(x,y)\in M\times N:~ d(f(x),y)\leq \varepsilon\}.
    \end{align*}
    For a $C^*$-algebra $\mathcal{A}$, an $\varepsilon$-cover of $f$ is an isometry
    $V\otimes \text{id}_{\mathcal{A}}:\ell^2(M_0;\mathcal{H}_{\mathcal{A}}) \to \ell^2(N_0;\mathcal{H}_{\mathcal{A}})$ for some $\varepsilon$-cover $V:\ell^2(M_0;\mathcal{H}) \to \ell^2(N_0;\mathcal{H})$ between Hilbert spaces.

    A uniformly continuous family $\{V_t\}_{t\in [1,\infty)}$ of isometries from $\ell^2(M_0;\mathcal{H}_{\mathcal{A}})$ to $\ell^2(N_0;\mathcal{H}_{\mathcal{A}})$ is called an continuous cover of $f$ if there exists a decreasing sequence $\{\varepsilon_k\}_k$ of positive numbers with $\lim_{k}\varepsilon_k=0$ such that $V_t$ is an $\varepsilon_k$ cover of $f$ for all $t\geq k$.
\end{definition}

Let $M$ and $N$ be metric spaces with bounded geometry. Fix their net $X\subset M$ and $Y\subset N$, and Borel covers $\{B_x\}_{x\in X}$ and $\{C_y\}_{y\in Y}$ of $M$ and $N$ satisfying the four conditions in Lemma~\ref{borel cover}, respectively. Let $f:M\to N$ be a continuous coarsely equivariant map. Fix any stable coarse $Y$-algebra $\Gamma(Y,\mathcal{A}\otimes \mathcal{K})$, and let $\Gamma(X,\mathcal{A}\otimes \mathcal{K})$ be the coarse $X$-algebra induced by $f$. 

We can also take nets $X'\subset M$ and $Y'\subset N$ and Borel covers $\{B'_x\}_{x\in X'}$ and $\{C'_y\}_{y\in Y'}$ of $M$ and $N$ such that $f$ is a bijection between $X'$ and $Y'$ and $f(B'_x)\subset C'_{f(x)}$.
We fix a decreasing sequence $\{\varepsilon_j\}_j$ of positive numbers with $\lim_{j}\varepsilon_j=0$.
For each $x\in X'$, since $f$ is continuous, by decomposing $C'_{f(x)}$ into a smaller disjoint Borel sets satisfying the condition in Lemma~\ref{borel cover}, we obtain a unitary
\begin{align*}
    V_j(x):\ell^2(B'_{x}\cap M_0,\mathcal{H})\to \ell^2(C'_{f(x)}\cap N_0,\mathcal{H})
\end{align*}
such that $\supp V_j(x)\subset \{(p,q)\in B'_{x}\times C'_{f(x)}:d(f(p),q)\leq \varepsilon_j\}$. We consider
\begin{align*}
    V_j:=\bigoplus_{x\in X'} V_j(x)\otimes \text{id}_{\mathcal{A}} : \ell^2( M_0,\mathcal{H}_{\mathcal{A}})\to \ell^2(N_0,\mathcal{H}_{\mathcal{A}}),
\end{align*}
then $V_j$ is an $\varepsilon_j$-cover of $f$. We denote by $\Gamma(X',\mathcal{A}\otimes \mathcal{K})$ and $\Gamma(Y',\mathcal{A}\otimes \mathcal{K})$ the coarse algebra induced by the coarse equivalence between $X$ and $X'$, and between $Y$ and $Y'$, respectively. Then, it is obvious that the adjoint by $V_j$ sends $C^*(M;\Gamma(X',\mathcal{A}\otimes \mathcal{K}))$ to $C^*(N;\Gamma(Y',\mathcal{A}\otimes \mathcal{K}))$. By Lemma~\ref{independence of nets}, for the originally fixed $\{B_x\}_{x\in X}$ and $\{C_y\}_{y\in Y}$, the adjoint by $V_j$ sends $C^*(M;\Gamma(X,\mathcal{A}\otimes \mathcal{K}))$ to $C^*(N;\Gamma(Y,\mathcal{A}\otimes \mathcal{K}))$. Define a uniformly continuous family of unitaries $(W(t))_{t\in [1,\infty)}$, where
    \begin{align*}
        W(t):\ell^2\left(M_0,\mathcal{H}\otimes\mathbb{C}^2\otimes \mathcal{A}\right)\to \ell^2\left(N_0,\mathcal{H}\otimes\mathbb{C}^2\otimes \mathcal{A}\right) 
    \end{align*}
    by
    \begin{align}\label{continuous covers}
        W(t)=R(t-j+1)(V_j\oplus V_{j+1})R(t-j+1)^*
    \end{align}
    for all $j-1\leq t\leq j$, where $R(t)$ is the rotation
    \begin{align*}
        R(t)=\begin{pmatrix}
            \cos{(\pi t/2)} & \sin{(\pi t/2)}\\
            -\sin{(\pi t/2)} & \cos{(\pi t/2)}\\
        \end{pmatrix}.
    \end{align*}
    Then, $(W(t))_t$ is a continuous cover of $f$ and the adjoint by $(W(t))_t$ intertwines $C^*_L(M;\Gamma(X,\mathcal{A}\otimes \mathcal{K}))$ and $C^*_L(N;\Gamma(Y,\mathcal{A}\otimes \mathcal{K}))$. We denote
    \begin{align*}
        f_*:=\text{ad}(W(t))_*:K_*(C^*_L(M;\Gamma(X,\mathcal{A}\otimes \mathcal{K}))) \to K_*(C^*_L(N;\Gamma(Y,\mathcal{A}\otimes \mathcal{K}))).
    \end{align*}
    \begin{lemma}
        The map $f_*$ is independent of the choice of cover $\{W_t\}$ of $f$ such that
        \begin{align*}
            \text{ad}(W_t)\left(C^*_L(M;\Gamma(X,\mathcal{A}\otimes \mathcal{K}))\right) \subset C^*_L(N;\Gamma(Y,\mathcal{A}\otimes \mathcal{K})).
        \end{align*}
    \end{lemma}
    \begin{proof}
Assume that families of isometries $\{W_t^1\}$ and $\{W_t^2\}$ are covers of $f$.
Since the unitary
\begin{align*}
U(t)
:=
\begin{pmatrix}
1 - W_t^1 (W_t^1)^* & W_t^1 (W_t^2)^* \\
W_t^2 (W_t^1)^* & 1 - W_t^2 (W_t^2)^*
\end{pmatrix}
\end{align*}
intertwines 
\[
\begin{pmatrix}
W_t^1 T (W_t^1)^* & 0 \\
0 & 0
\end{pmatrix}
\quad \text{and} \quad
\begin{pmatrix}
0 & 0 \\
0 & W_t^2 (W_t^2)^*
\end{pmatrix},
\]
it suffices to show that the function $\{U(t)\}$ is in the multiplier of 
$C^*_L(N;\Gamma(Y,\mathcal{A}\otimes \mathcal{K})) \otimes M_2$.
It is easy to see that $\{W_t^i (W_t^j)^*\}$ is in the multiplier of 
$C^*_L(N)$ for $i,j=1,2$.

For every partial translation $v$ on $Y$, $(T_t)\in C^*_L(N;\Gamma(Y,\mathcal{A}\otimes \mathcal{K}))$ and $t\in [1,\infty)$, it suffices to show that the map 
\[
\left( T_t W_t^i (W_t^j)^* \right)^v : Y \to \mathcal{A}\otimes \mathcal{K}
\]
is in $\Gamma(Y,\mathcal{A}\otimes \mathcal{K})$.
We have that
\begin{align*}
\left( T_t W_t^i (W_t^j)^* \right)^v (y)
&= \chi_{C_y} \left( T_t W_t^i (W_t^j)^* \right) \chi_{C_{v(y)}} \\
&= \sum_z 
\left( \chi_{C_y} T_t \chi_{C_z} \right)
\left( \chi_{C_z} W_t^i (W_t^j)^* \chi_{C_{v(y)}} \right) \\
&= \sum_z 
\left( \chi_{C_y} T_t \chi_{C_z} \right)
\left(
\chi_{C_z}
\bigl(
\tilde{W}_t^i (\tilde{W}_t^j)^* \otimes \mathrm{id}_{\mathcal A}
\bigr)
\chi_{C_{v(y)}}
\right).
\end{align*}
for some covers $\{\tilde{W}_t^i\}$ and $\{\tilde{W}_t^j\}$ between Hilbert spaces 
$\ell^2(M_0,\mathcal{H})$ and $\ell^2(N_0,\mathcal{H})$.
By the third condition in Definition~\ref{stable coarse algebra}, we obtain that
\[
\left( T_t W_t^i (W_t^j)^* \right)^v
\]
is in $\Gamma(Y,\mathcal{A}\otimes \mathcal{K})$.
This finishes the proof.
\end{proof}

    Next, we discuss the coarse invariance of the twisted coarse Baum--Connes conjecture with respect to stable coarse algebras. The following definition is from $\cite{Yu1997LocalizationandCoarseBaumConnes}$.
    \begin{definition}
        Let $f,g:M \to N$ be continuous maps between metric spaces $M$ and $N$ with bounded geometry. A continuous homotopy $\{F(t,\cdot)\}_{t\in [0,1]}$ between $f$ and $g$ is said to be strongly Lipschitz if
        \begin{enumerate}[(1)]
            \item for each $t$, $F(t,\cdot):X \to Y$ is a continuous map;

            \item there exists a constant $C>0$ such that $d(F(t,x),F(t,y))\leq Cd(x,y)$ for all $x,y\in X$ and $t\in [0,1]$;

            \item the family $\{F(\cdot,x)\}_{x\in X}$ is uniformly equicontinuous;

            \item $F(0,\cdot)=f$, $F(1,\cdot)=g$.
        \end{enumerate}
    \end{definition}

\begin{lemma}\label{homotopy}
    Let $X$ and $Y$ be metric spaces with bounded geometry. Assume that there exist coarse equivalences $f:X\to Y$ and $g:Y\to X$ such that the compositions $f\circ g$ and $g\circ f$ are close to the identity maps ${\rm id}_Y$ and ${\rm id}_X$, respectively. For any coarse $Y$-algebra $\Gamma(Y,\mathcal{A}\otimes\mathcal{K})$, if $\Gamma(X,\mathcal{A}\otimes\mathcal{K})$ is the coarse $X$-algebra induced the coarse equivalence $f$ via the formula \eqref{twist for coarse equi}, then we have that
    \begin{align*}
        \varinjlim_s K_*(C^*_L(P_s(X);\Gamma(X,\mathcal{A}\otimes \mathcal{K}))) \cong \varinjlim_s K_*(C^*_L(P_s(Y);\Gamma(Y,\mathcal{A}\otimes \mathcal{K}))).
    \end{align*}
\end{lemma}
\begin{proof}
    We show that $g_*\circ f_*$ is the identity on $\varinjlim_s K_*(C^*_L(P_s(X);\Gamma(X,\mathcal{A}\otimes \mathcal{K})))$.
    Note that for every $s>0$ there exists $s'>0$ such that $f$ and $g$ can be linearly extended to continuous maps
    $f:P_s(X)\to P_{s'}(Y)$ and $g:P_s(Y)\to P_{s'}(X)$ such that their compositions are strongly Lipschitz homotopic to the inclusions $P_s(X)\to P_{s''}(X)$ and $P_s(Y)\to P_{s''}(Y)$.
    Take continuous covers
    \begin{align*}
        V_f(t)&: \ell^2(P_s(X)_0,\mathcal{H}_{\mathcal{A}})\to \ell^2(P_{s'}(Y)_0,\mathcal{H}_{\mathcal{A}}),\\
        V_g(t)&: \ell^2(P_{s'}(Y)_0,\mathcal{H}_{\mathcal{A}})\to \ell^2(P_{s''}(X)_0,\mathcal{H}_{\mathcal{A}})
    \end{align*}
    which intertwines the twisted localization algebras. Since the family $\{V_g(t)V_f(t)\}_t$ covers $g\circ f$, it suffices to show that
    \begin{align*}
        (g\circ f)_*={\iota_{s,s''}}_*:K_*(C^*_L(P_s(X);\Gamma(X,\mathcal{A}\otimes \mathcal{K}))) \to K_*(C^*_L(P_{s''}(X);\Gamma(X,\mathcal{A}\otimes \mathcal{K}))),
    \end{align*}
    where $\iota_{s,s''}$ is the natural inclusion $P_s(X)\to P_{s''}(X)$. Let $F(t,x)$ be a strong Lipschitz homotopy between $g\circ f$ and $\iota_{s,s''}$ ($F(0,\cdot)=g\circ f$ and $F(1,\cdot)=\iota_{s,s''}$). Take a double sequence $\{t_{i,j}\}_{i,j}$ of real numbers in $[0,1]$, a decreasing sequence $\{\varepsilon_j\}$ of positive numbers such that $\lim \varepsilon_j=0$ and a sequence of natural numbers $\{N_j\}$ such that
        \begin{enumerate}[(1)]
    \item  $t_{0,j}=0$ for all $j$ and $t_{i,j}=1$ if $i\geq N_j$;

    \item $d(F(t_{i,j},\cdot)(x),F(t_{i+1,j},\cdot)(x))\leq \varepsilon_j$;

    \item $d(F(t_{i,j},\cdot)(x),F(t_{i,j+1},\cdot)(x))\leq \varepsilon_j$.
\end{enumerate}
For each $i$ and $j$, take $\varepsilon_j$-covers $V_{i,j}$ of $F(t_{i,j},\cdot)$ and define a family $\{W_i(t)\}_{t\in [1,\infty)}$ by the same formula as \eqref{continuous covers}. It suffices to show that $[\text{ad}(W_0)(u)]=[u]\in K_1(C^*_L(P_{s''}(X);\Gamma(X,\mathcal{A}\otimes \mathcal{K})))$ for any unitary $u$ in the unitaization of $C^*_L(P_s(X);\Gamma(X,\mathcal{A}\otimes \mathcal{K}))$. (We extend $u$ to the identity on the complement $\ell^2(P_{s''}(X)_0\setminus P_s(X)_0,\mathcal{H}_{\mathcal{A}})$ to regard it as an unitary in the unitaization of $C^*_L(P_{s''}(X);\Gamma(X,\mathcal{A}\otimes \mathcal{K}))$. Define unitaries $a,b,c$ acting on $\ell^2(P_{s''}(X)_0,\ell^2\otimes \mathcal{H}_{\mathcal{A}})$ by 
\begin{align*}
    a&=\bigoplus_{i=0}^{\infty}\text{ad}(W_i)(u) u^*,\\
b&=\bigoplus_{i=0}^{\infty}\text{ad}(W_{i+1})(u) u^*,\\
c&=I\oplus \bigoplus_{i=1}^{\infty} \text{ad}(W_i)(u) u^*.
\end{align*}
Following \cite[Proposition~3.7]{Yu1997LocalizationandCoarseBaumConnes}, these are in the unitaization of $C^*_L(P_{s''}(X))$. In addition, since 
$$a(t)-I=\text{diag}(\text{ad}(W_0)(u)u^*-I,\cdots, \text{ad}(W_N)(u)u^*-I,0,0\cdots),$$ 
by the fourth condition in Definition~\ref{stable coarse algebra}, $a-I$ is an element in $C^*_L(P_{s''}(X);\Gamma(X,\mathcal{A}\otimes \mathcal{K}))$. Similarly, $b-I$ and $c-I$ are elements in $C^*_L(P_{s''}(X);\Gamma(X,\mathcal{A}\otimes \mathcal{K}))$. By arguing as \cite[Proposition~3.7]{Yu1997LocalizationandCoarseBaumConnes}, we can show $[a]=[c]$ and this implies $[\text{ad}(W_0)(u)]=[u]\in K_1(C^*_L(P_{s''}(X);\Gamma(X,\mathcal{A}\otimes \mathcal{K})))$. The other composition $f_*\circ g_*$ can be shown to be the identity on $\varinjlim_s K_*(C^*_L(P_s(Y);\Gamma(Y,\mathcal{A}\otimes \mathcal{K})))$ by the same argument.
\end{proof}

Combining the above result with Theorem~\ref{coarse equivalence}, we obtain the following permanence result.

\begin{corollary}
    Let $X$ and $Y$ be metric spaces with bounded geometry. Assume that $X$ and $Y$ are coarsely equivalent. Then $X$ satisfies the twisted coarse Baum--Connes conjecture with respect to any coarse algebras if and only if $Y$ does. 
\end{corollary}

Next, we prove the permanence of the twisted coarse Baum--Connes conjecture with respect to any stable coarse algeras under taking subspaces.

\begin{theorem}\label{permanence for subsp.}
    Let $Y\subset X$ be metric spaces with bounded geometry. For any stable coarse 
 $Y$-algebra $\Gamma(Y,\mathcal{A}\otimes \mathcal{K})$, let $\Gamma(X,\mathcal{A}\otimes\mathcal{K})$ be the stable coarse $X$-algebra defined in \eqref{coarse alg from subspace}. Then the natural inclusion
 \begin{align*}
     C^*_L(P_s(Y);\Gamma(Y,\mathcal{A}\otimes \mathcal{K})) \to C^*_L(P_s(X);\Gamma(X,\mathcal{A}\otimes \mathcal{K}))
 \end{align*}
 induces an isomorphism
 \begin{align*}
     \varinjlim_{s}K_*(C^*_L(P_s(Y);\Gamma(Y,\mathcal{A}\otimes \mathcal{K}))) \cong \varinjlim_{s}K_*(C^*_L(P_s(X);\Gamma(X,\mathcal{A}\otimes \mathcal{K}))).
 \end{align*}
Moreover, if $X$ satisfies the twisted coarse Baum--Connes conjecture with respect to any stable coarse algebras, then so does $Y$.
\end{theorem}
Before proving this, we need following preliminary Definition and Lemma.
\begin{definition}\label{uniform twisted}
    Let $M$ be any metric space with bounded geometry containing a net $X$, and $Y\subset X$ a subset of $X$. Let $\Gamma(Y,\mathcal{A}\otimes \mathcal{K})$ be any stable coarse $Y$-algebra and $\Gamma(N_r(Y),\mathcal{A}\otimes\mathcal{K})$ and $\Gamma(X,\mathcal{A}\otimes\mathcal{K})$ are defined in the proof of Theorem~\ref{twist for subsp.}. We define a $C^*$-algebra $C^*_{L,\text{uniform}}(M;\Gamma(X,\mathcal{A}\otimes \mathcal{K}))$
    to be the closure of the $*$-algebra consisting all $(T_t)_t$ in $\mathbb{C}_L[M;\Gamma(X,\mathcal{A}\otimes \mathcal{K})]$ such that for any partial translation $v$ there exists $r>0$ such that $(T_t)^v\in \Gamma(N_r(Y),\mathcal{A}\otimes \mathcal{K})$ for all $t\in [1,\infty)$.
\end{definition}

\begin{lemma}
    Under the setting of Definition~\ref{uniform twisted}, if $M$ is a finite-dimensional simplicial complex, then we have an isomorphism
    \begin{align*}
        K_*(C^*_{L,\text{uniform}}(M;\Gamma(X,\mathcal{A}\otimes \mathcal{K})))\cong K_*(C^*_{L}(M;\Gamma(X,\mathcal{A}\otimes \mathcal{K})))
    \end{align*}
    induced by the natural inclusion $$C^*_{L,\text{uniform}}(M;\Gamma(X,\mathcal{A}\otimes \mathcal{K})) \to C^*_{L}(M;\Gamma(X,\mathcal{A}\otimes \mathcal{K})).$$
\end{lemma}

\begin{proof}
    We prove this by the induction on the dimension $m$ of $M$. In the case when $m=0$, we may assume $M=X$ and it suffices to show that the inclusion
    \begin{align*}
        \varinjlim_r C_{ub}([1,\infty),\Gamma(N_r(Y),\mathcal{A}\otimes \mathcal{K})) \to C_{ub}([1,\infty),\Gamma(X,\mathcal{A}\otimes \mathcal{K}))
    \end{align*}
    induces an isomorphism on $K$-theory, where for any $C^*$-algebra $\mathcal{B}$, $C_{ub}([1,\infty),\mathcal{B})$ is the $C^*$-algebra of all uniformly continuous, bounded function from $[1,\infty)$ to $\mathcal{B}$. But since $\Gamma(N_r(Y),\mathcal{A}\otimes \mathcal{K}))$ and $\Gamma(X,\mathcal{A}\otimes \mathcal{K}))$ are quasi-stable, the evaluation-at-one maps
    \begin{align*}
        C_{ub}([1,\infty),\Gamma(N_r(Y),\mathcal{A}\otimes \mathcal{K})) &\to \Gamma(N_r(Y),\mathcal{A}\otimes \mathcal{K})\\ 
        C_{ub}([1,\infty),\Gamma(X,\mathcal{A}\otimes \mathcal{K})) &\to \Gamma(X,\mathcal{A}\otimes \mathcal{K}))
    \end{align*}
    induce isomorphisms on $K$-theory by \cite[Lemma~12.4.3]{WillettYuHigherindextheoryBook}. The assertion follows from the continuity of $K$-theory. Next, we assume induction hypothesis for the dimension $0,\cdots, m-1$. For each $m$-dimensional simplicial complex $\Delta$, we denote by $c(\Delta)$ the center of it. We define
    \begin{align*}
        \Delta_1:=\left\{x\in \Delta:d(x,c(\Delta))\leq \frac{1}{10}\right\},\quad \Delta_2:=\left\{x\in \Delta:d(x,c(\Delta))\geq \frac{1}{10}\right\}
    \end{align*}
    and define 
    \begin{align*}
        N_1':=\bigcup \{\Delta_1&:\Delta \text{ is an $m$-dimensional simplex of }M\}\\ &\cup \bigcup \{\Delta:\Delta \text{ is a simplex of}~M~\text{whose dimension is less than $m$}\}\\
        N_2':=\bigcup \{\Delta_2&:\Delta \text{ is an $m$-dimensional simplex of }M\}\\ &\cup \bigcup \{\Delta:\Delta \text{ is a simplex of}~M~\text{whose dimension is less than $m$}\}.
    \end{align*}
    Taking a point from each simplex, we fix a net $X'_1\subset N_1'$, $X'_2\subset N_2'$ and $X_{1,2}'\subset N'_1\cap N'_2$, respectively. Since $N_1'$, $N_2'$ and $N_1'\cap N_2'$ are strongly Lipschitz homotopic to simplicial complexes $N_1$, $N_2$ and $N_{1,2}$ whose dimensions are less than $m$, respectively through coarse equivalent maps. We denote the corresponding nets $X_1\subset N_1$, $X_2\subset N_2$ and $X_{1,2}\subset N_{1,2}$, respectively. By the same argument as Lemma~\ref{homotopy}, one can show that 
    \begin{align*}
        K_*(C^*_{L,\text{uniform}}(N_i';\Gamma(X_i',\mathcal{A}\otimes \mathcal{K}))) &\cong K_*(C^*_{L,\text{uniform}}(N_i;\Gamma(X_i,\mathcal{A}\otimes \mathcal{K}))),\\
        K_*(C^*_{L}(N_i';\Gamma(X_i',\mathcal{A}\otimes \mathcal{K}))) &\cong K_*(C^*_{L}(N_i;\Gamma(X_i,\mathcal{A}\otimes \mathcal{K})))
    \end{align*}
     for $i=1,2$ and 
    \begin{align*}
        K_*(C^*_{L,\text{uniform}}(N_1'\cap N_2';\Gamma(X_{1,2}',\mathcal{A}\otimes \mathcal{K}))) &\cong K_*(C^*_{L,\text{uniform}}(N_{1,2};\Gamma(X_{1,2},\mathcal{A}\otimes \mathcal{K})))\\
        K_*(C^*_{L}(N_1'\cap N_2';\Gamma(X_{1,2}',\mathcal{A}\otimes \mathcal{K}))) &\cong K_*(C^*_{L}(N_{1,2};\Gamma(X_{1,2},\mathcal{A}\otimes \mathcal{K}))).
    \end{align*}
    Note that $K$-theory for both $C^*_{L,\text{uniform}}(M;\Gamma(X,\mathcal{A}\otimes \mathcal{K})))$ and $C^*_{L}(M;\Gamma(X,\mathcal{A}\otimes \mathcal{K})))$ admits Mayer-Vietoris sequence by the same proof as \cite[Proposition~3.11]{Yu1997LocalizationandCoarseBaumConnes}. Therefore, the assertion follows from the five lemma and inductive hypothesis.
\end{proof}

\begin{proof}[Proof of Theorem~\ref{permanence for subsp.}]
    We construct a map 
    \begin{align*}
        K_1(C^*_{L,\text{uniform}}(P_s(X);\Gamma(X,\mathcal{A}\otimes \mathcal{K}))) \to \varinjlim_{s}K_1(C^*_L(P_s(Y);\Gamma(Y,\mathcal{A}\otimes \mathcal{K}))).
    \end{align*}
    which is the inverse of a map induced by the inclusion
     \begin{align}\label{inclusion}
     C^*_L(P_s(Y);\Gamma(Y,\mathcal{A}\otimes \mathcal{K})) \to C^*_{L,\text{uniform}}(P_s(X);\Gamma(X,\mathcal{A}\otimes \mathcal{K})).
 \end{align}
 Take any unitary in the unitaization of $C^*_{L,\text{uniform}}(P_s(X);\Gamma(X,\mathcal{A}\otimes \mathcal{K}))$. We may assume that there exists $R\geq 0$ and a scalar $\lambda\in \mathbb{C}$ such that $u_t-\lambda$ is supported in $(P_s(N_R(Y)))\times (P_s(N_R(Y)))$ for every $t$. Therefore we can regard $[u_t]$ as an element in $K_*(C^*_L(P_s(N_R(Y));\Gamma(Y,\mathcal{A}\otimes \mathcal{K})))$. Since we have an isomorphism 
 \begin{align*}
     \varinjlim_s K_*(C^*_L(P_s(N_R(Y));\Gamma(Y,\mathcal{A}\otimes \mathcal{K})))\cong \varinjlim_s K_*(C^*_L(P_s(Y));\Gamma(Y,\mathcal{A}\otimes \mathcal{K})))
 \end{align*}
 by Lemma~\ref{homotopy}, we obtain an element in $\varinjlim_s K_*(C^*_L(P_s(Y));\Gamma(Y,\mathcal{A}\otimes \mathcal{K})))$. It is obvious that this is an inverse to the map induced by \eqref{inclusion}. 
\end{proof}

\subsection{Family of subspaces}
In this subsection, for a given metric space $X$ with bounded geometry and a family of subspaces $\{Y^{(j)}\}_j$ of $X$, we study the connection between the twisted coarse Baum--Connes conjecture of the family $\{Y^{(j)}\}_j$ and that of the space $X$.

We first realize the $K$-theory of twisted Roe algebras of $\bigsqcup_j Y^{(j)}$ as that of $X$ by using a suitable coarse algebra. We denote by $\Tilde{d}$ the distance on $\bigsqcup_j Y^{(j)}$ defined  by
\begin{align*}
    \Tilde{d}(x,y)=\left\{\begin{array}{cc}
        d(x,y), & \mbox{if~} x, y\in Y^{(j)} \text{ for some }j,   \\
         \infty, & \text{otherwise.}
    \end{array} \right.
\end{align*}
We view $(\bigsqcup_j Y^{(j)},\Tilde{d})$ as a subspace of $\bigsqcup_j X$ (the distance between different copies of $X$ is taken to be infinity again). Assume that $\Gamma(\bigsqcup_j Y^{(j)},\mathcal{A}\otimes \mathcal{K})$ is a stable coarse $\bigsqcup_j Y^{(j)}$-algebra.
By Theorem \ref{twist for subsp.}, we have that
\begin{align*}
    C^*\left(\bigsqcup_j X;\Gamma\left(\bigsqcup_j X,\mathcal{A}\otimes \mathcal{K}\right)\right) \cong C^*\left(\bigsqcup_j Y^{(j)};\Gamma\left(\bigsqcup_j Y^{(j)},\mathcal{A}\otimes \mathcal{K}\right)\right),
\end{align*}
where $\Gamma(\bigsqcup_j X,\mathcal{A}\otimes \mathcal{K})$ is constructed from the coarse  $\bigsqcup Y^{(j)}$-algebra $\Gamma(\bigsqcup_j Y^{(j)};\mathcal{A}\otimes \mathcal{K})$ and the inclusion $\bigsqcup Y^{(j)}\subset \bigsqcup X$ by the formula \eqref{coarse alg from subspace}. 

\begin{remark}
Let $\mathcal{A}$ be a $C^*$-algebra, $M$ a metric space with bounded geometry and a countable dense subset $M_0$ and a net $X\subset M$.
    For isometries $V:\mathcal{H}_\mathcal{A}\to \ell^2\otimes \mathcal{H}_\mathcal{A}\cong \mathcal{H}_\mathcal{A},~ \xi\mapsto (\xi,0,0,\cdots)$, and $W:\ell^2\otimes \mathcal{H}_\mathcal{A}\cong \mathcal{H}_\mathcal{A}\to \ell^2\otimes\mathcal{H}_\mathcal{A},~ \xi\mapsto (\xi,0,0,\cdots)$, we denote the isometries induced on the geometric modules
    \begin{align*}
        \Tilde{V}:\ell^2(M_0,\mathcal{H}_\mathcal{A}) &\to \ell^2(M_0,\mathcal{H}_\mathcal{A}),\\
        \Tilde{W}:\ell^2(M_0,\ell^2\otimes\mathcal{H}_\mathcal{A}) &\to \ell^2(M_0,\ell^2\otimes\mathcal{H}_\mathcal{A}).
    \end{align*}
    For any stable coarse $X$-algebra $\Gamma(X,\mathcal{A}\otimes \mathcal{K})$, the adjoint by $\Tilde{V}$ and $\Tilde{W}$ can be restricted to the twisted Roe algebra $C^*(M,\Gamma(X,\mathcal{A}\otimes \mathcal{K}))$ by the fourth condition of Definition~\ref{stable coarse algebra} and the induced map on $K$-theory is the identity map by the same argument in the proof as \cite[Lemma 5.1.12]{WillettYuHigherindextheoryBook}.
\end{remark}

\begin{lemma}\label{uniform coarse BaumConnes of subspaces}
Let $M$ be a metric space with bounded geometry and fix a net $X\subset M$.
    Given a coarse $\bigsqcup_j X$-algebra $\Gamma(\bigsqcup_j X,\mathcal{A}\otimes \mathcal{K})$, we define a new stable coarse $X$-algebra with respect to $\ell^{\infty}(\mathbb{N},\mathcal{A}\otimes \mathcal{K} )\otimes\mathcal{K}$ as
    \begin{align*}
        \Gamma(X;\ell^{\infty}(\mathbb{N},\mathcal{A}\otimes\mathcal{K})\otimes \mathcal{K}):=\left\{f\in \ell^{\infty}(X,\ell^{\infty}(\mathbb{N},\mathcal{A}\otimes\mathcal{K})\otimes\mathcal{K}):\left(\psi(f_{\bullet}(j))\right)_j\in \Gamma\left(\bigsqcup_j X,\mathcal{A}\otimes \mathcal{K}\right)\right\},
    \end{align*}
    where $\psi:\mathcal{A}\otimes \mathcal{K}\otimes\mathcal{K}\xrightarrow{\cong}\mathcal{A}\otimes\mathcal{K}$ is the fixed isomorphism.
    Then we have an isomorphism for Roe algebras and localization algebras on $K$-theory:
    \begin{align*}
       K_*\left( C^*\left(\bigsqcup_j M;\Gamma\left(\bigsqcup_j X,\mathcal{A}\otimes \mathcal{K}\right)\right)\right)&\cong K_*\left(C^*\left(M;\Gamma(X;\ell^{\infty}(\mathbb{N},\mathcal{A}\otimes\mathcal{K})\otimes \mathcal{K})\right)\right)\\
        K_*\left( C_{L}^*\left(\bigsqcup_j M;\Gamma\left(\bigsqcup_j X,\mathcal{A}\otimes \mathcal{K}\right)\right)\right)&\cong K_*\left(C_{L}^*\left(M;\Gamma(X;\ell^{\infty}(\mathbb{N},\mathcal{A}\otimes\mathcal{K})\otimes \mathcal{K})\right)\right).
    \end{align*}
\end{lemma}

\begin{proof} 
Define $*$-homomorphisms 
    \begin{align*}
        \alpha:&C^*\left(\bigsqcup_j X;\Gamma\left(\bigsqcup_j X,\mathcal{A}\otimes \mathcal{K}\right)\right) \to C^*\left(X;\Gamma(X,\ell^{\infty}(\mathbb{N},\mathcal{A}\otimes\mathcal{K})\otimes \mathcal{K})\right),\\
        \beta:&C^*\left(X;\Gamma(X,\ell^{\infty}(\mathbb{N},\mathcal{A}\otimes\mathcal{K})\otimes \mathcal{K})\right) \to C^*\left(\bigsqcup_j X;\Gamma\left(\bigsqcup_j X,\mathcal{A}\otimes \mathcal{K}\right)\right)
    \end{align*}
    by
    \begin{align*}
        \alpha((T^{(j)}))_{x,y} &:=(j\mapsto T^{(j)}_{x,y})\otimes e_{1,1}\\
        \beta(S)^{(j)}_{x,y}&:=\psi(S_{x,y}(j)),
    \end{align*}
    where $e_{1,1}\in \mathcal{K}$ is any rank one projection.
The compositions $\beta\circ\alpha$ and $\alpha \circ \beta$ are the adjoints by $\Tilde{V}$ and $\Tilde{W}$ in the previous remark, respectively, and both induce isomorphisms on $K$-theory. The same proof works for localization algebras.
\end{proof}

Using the above lemma, we are able to reduce the twisted coarse Baum--Connes conjecture with respect to any stable coarse algebras for $\bigsqcup X$ to that of $X$.

\section{Filtered $C^*$-algebra and controlled $K$-theory}

In this section, we recall the concepts of filtered $C^*$-algebras and controlled $K$-theory from \cite{GuentnerWillettYu2024DynamicalComplexity}. Using these techniques, we obtain a controlled Mayer--Vietoris argument for the twisted coarse Baum--Connes conjecture.  

In order to use $C^*$-algebras to encode geometry of metric spaces, we need the concept of filtration. 
\begin{definition}
    A filtration on a $C^*$-algebra $A$ is a collection of self-adjoint subspaces $(A_r)_{r\geq 0}$ of $A$ indexed by the non-negative real numbers with the following properties:
    \begin{enumerate}[(1)]
        \item if $r_1\leq r_2$, then $A_{r_1}\subset A_{r_2}$;

        \item for all $r_1,r_2$, we have $A_{r_1}\cdot A_{r_2}\subset A_{r_1\cdot r_2}$;

        \item the union $\bigcup_{r\geq 0} A_r$ is dense in $A$.
        
    \end{enumerate}
\end{definition}
In the case of Roe algebras, the filtration is defined by propagations. We apply this framework to the obstruction algebras $C^*_{L,0}(M;\Gamma(X,\mathcal{A}))$ defined in Definition~\ref{twisted coarse baum--connes}. For $r>0$, we define
\begin{align*}
    C^*_{L,0}(M;\Gamma(X,\mathcal{A}\otimes \mathcal{K}))_r:=\{f\in C^*_{L,0}(M;\Gamma(X,\mathcal{A}\otimes \mathcal{K})):\text{propagation}(f(t))<r~\forall t\in [1,\infty)\}.
\end{align*}

Next, we shall recall the quantitative $K$-theory introduced by G. Yu \cite{Yu1998NovikoforFiniteAsymptoticDimension} to prove the coarse Baum-Connes conjecture for metric spaces with finite asymptotic dimension. 
\begin{definition}\label{quasiprojection}
Let $A$ be a $C^*$-algebra. A \emph{quasi-projection} in $A$ is an element $p \in A$ such that $p = p^*$ and $\|p^2 - p\| < \frac{1}{8}$. 

If $S$ is a self-adjoint subspace of $A$, write $\mathrm{M}_n(S)$ for the matrices in $\mathrm{M}_n(A)$ with all entries coming from $S$, and let $P^{1/8}_n(S)$ denote the collection of quasi-projections in $\mathrm{M}_n(S)$.

Let $\chi = \chi_{(1/2, \infty)}$ be the characteristic function of the interval $(1/2, \infty)$. Then $\chi$ is continuous on the spectrum of any quasi-projection, and thus there is a well-defined map:
\[
\kappa : P^{1/8}_n(S) \to P_n(A), \quad p \mapsto \chi(p),
\]
where $P_n(A)$ denotes the projections in $\mathrm{M}_n(A)$.
\end{definition}

\begin{definition}
Let $A$ be a non-unital $C^*$-algebra with a filtration $(A_r)_r$. Let $\Tilde{A}_r$ denote the subspace $A_r + \mathbb{C}1$ of $\Tilde{A}$.

Using the inclusions
\[
P^{1/8}_n(\Tilde{A}_r) \ni p \mapsto
\begin{pmatrix}
p & 0 \\
0 & 0
\end{pmatrix}
\in P^{1/8}_{n+1}(\Tilde{A}_r),
\]
we define the union
\[
P^{1/8}_{\infty}(\Tilde{A}_r) := \bigcup_{n=1}^\infty P^{1/8}_n(\Tilde{A}_r).
\]

Let $C((0,1], \Tilde{A}_r)$ denote the self-adjoint subspace of the $C^*$-algebra $C((0,1], \Tilde{A})$ consisting of continuous functions with values in $\Tilde{A}_r$.

Define an equivalence relation on $P^{1/8}_{\infty}(\Tilde{A}_r) \times \mathbb{N}$ by declaring $(p, m) \sim (q, n)$ if there exists a positive integer $k$ and an element $h \in P^{1/8}_{\infty}(C((0,1], \Tilde{A}_r))$ such that
\[
h(0) =
\begin{pmatrix}
p & 0 \\
0 & 1_{n+k}
\end{pmatrix}, \quad
h(1) =
\begin{pmatrix}
q & 0 \\
0 & 1_{m+k}
\end{pmatrix}.
\]
For $(p, m) \in P^{1/8}_{\infty}(\Tilde{A}_r) \times \mathbb{N}$, denote by $[p, m]$ its equivalence class under $\sim$.

Let now $\rho: M_n(\Tilde{A}_r) \to M_n(\mathbb{C})$ be the restriction to $M_n(\Tilde{A})$ of the map induced on matrices by the canonical unital $*$-homomorphism $\rho: \Tilde{A}_r \to \mathbb{C}$ with kernel $A$.

Finally, define
\[
K^{r,1/8}_0(A) := \left\{ [p, m] \in P_{1/8}^\infty(\Tilde{A}_r) \times \mathbb{N} \, \middle| \, \mathrm{rank}(\kappa(\rho(p))) = m \right\} / \sim.
\]
The set $K^{r,1/8}_0(A)$ is equipped with an operation defined by
\[
[p, m] + [q, n] := \left[
\begin{pmatrix}
p & 0 \\
0 & q
\end{pmatrix}, m+n
\right].
\]
\end{definition}

\begin{definition}\label{quasiunitary}
    Let $A$ be a unital $C^*$-algebra. A \emph{quasi-unitary} in $A$ is an element $u \in A$ such that
\[
\|1 - uu^*\| < \tfrac{1}{8} \quad \text{and} \quad \|1 - u^*u\| < \tfrac{1}{8}.
\]

If $S$ is a self-adjoint subspace of $A$ containing the unit, write $U^{1/8}_n(S)$ for the collection of quasi-unitaries in $M_n(S)$.

Note that since $\|1 - u^*u\| < \tfrac{1}{8} < 1$, the element $u^*u$ is invertible. Hence, there is a well-defined map
\[
\kappa : U^{1/8}_n(S) \to U_n(A), \quad u \mapsto u(u^*u)^{-1/2},
\]
where $U_n(A)$ denotes the group of unitaries in $M_n(A)$.

\end{definition}

\begin{definition}
Let $A$ be a non-unital $C^*$-algebra with a filtration $(A_r)_r$.
Using the inclusions
\[
U^{1/8}_n(\Tilde{A}_r) \ni u \mapsto
\begin{pmatrix}
u & 0 \\
0 & 1
\end{pmatrix}
\in U^{1/8}_{n+1}(\Tilde{A}_r),
\]
we may define the union
\[
U^{1/8}_{\infty}(\Tilde{A}_r) := \bigcup_{n=1}^\infty U^{1/8}_n(\Tilde{A}_r).
\]

Define an equivalence relation on $\mathcal{U}_{1/8}^\infty(\Tilde{A}_r)$ by declaring $u \sim v$ if there exists an element
\[
h \in U^{1/8}_\infty(C((0,1], \Tilde{A}_r))
\quad \text{such that} \quad h(0) = u \quad \text{and} \quad h(1) = v.
\]
For $u \in U^{1/8}_\infty(\Tilde{A}_r)$, denote by $[u]$ its equivalence class under $\sim$.
Finally, define
\[
K^{r,1/8}_1(A) := U^{1/8}_\infty(\Tilde{A}_r) \big/ \sim.
\]
The set $K^{1/8}_1(A)$ is equipped with a binary operation defined by
\[
[u] + [v] := \left[
\begin{pmatrix}
u & 0 \\
0 & v
\end{pmatrix}
\right].
\]

\end{definition}
The advantage of quantitative $K$-theory is that it is more computable compared with the usual operator $K$-theory. It also has connections with the usual $K$-theory.  

\begin{definition}
Let $A$ be a non-unital $C^*$-algebra, and let $S$ be a self-adjoint subspace of $A$. Let $\kappa$ be one of the maps from Definitions~\ref{quasiprojection} or Definition~\ref{quasiunitary} (it will be clear from context which is meant).

Define maps:
\[
c_0 : K^{r,1/8}_0(A) \longrightarrow K_0(A), \quad [p, m] \mapsto [\kappa(p)] - [1_m],
\]
\[
c_1 : K^{r,1/8}_1(A) \longrightarrow K_1(A), \quad [u] \mapsto [\kappa(u)],
\]
\[
c := c_0 \oplus c_1 : K^{r,1/8}_*(A) \longrightarrow K_*(A).
\]

We call $c_0$, $c_1$, and $c$ the \emph{comparison maps}.

\end{definition}

The map $c$ is asymptotically isomorphic in the following sense:

\begin{proposition}[Proposition 4.9~\cite{GuentnerWillettYu2024DynamicalComplexity}]
Let $A$ be a non-unital $C^*$-algebra with a filtration $(A_r)_r$. Then for any $x\in K_*(A)$, there exists $r>0$ and an element $y\in K^{r,\frac{1}{8}}(A)$ such that $c(y)=x$. Moreover if $x,y \in K_*^{s,\frac{1}{8}}(A)$ for some $s>0$ satisfies $c(x)=c(y)$, then there exists $r>0$ such that $x=y$ in $K_*^{r,\frac{1}{8}}(A)$.
\end{proposition}

In the following, we shall recall some techniques for the computation of quantitative $K$-theory formulated in \cite{GuentnerWillettYu2024DynamicalComplexity}. Those techniques are crucial to reduce to computation of the $K$-theory of twisted Roe algebras of relative hyperbolic group to that of peripheral subgroups.

\begin{definition}\label{filtered ideal}
 Let $A$ be a filtered $C^*$-algebra, and $I$ be a $C^*$-ideal in $A$ equipped with its own filtration. We say that $I$ is a \textit{filtered ideal} of $A$ if for any $r \geq 0$, $I_r \subset A_r$, and if for any $r, s \geq 0$, $A_s \cdot I_r \cup I_r \cdot A_s \subset I_{s+r}$.
\end{definition}

Now, we recall the following techniques to compute the quantitative $K$-theory. 

\begin{definition}\label{definition of excisive}
Let $(I^\omega, J^\omega; A^\omega)_{\omega \in \Omega}$ be an indexed set of triples, where each $A^\omega$ is a filtered $C^*$-algebra, and each $I^\omega$ and $J^\omega$ is a filtered ideal in $A^\omega$. Give each stabilization $A^\omega \otimes \mathcal{K}$, $I^\omega \otimes \mathcal{K}$ and $J^\omega \otimes \mathcal{K}$ the filtration 
\begin{align*}
    (A^\omega \otimes \mathcal{K})_r=A^\omega_r \otimes \mathcal{K},~(I^\omega \otimes \mathcal{K})_r=I^\omega_r \otimes \mathcal{K},~(J^\omega \otimes \mathcal{K})_r=J^\omega_r \otimes \mathcal{K}.
\end{align*}

\medskip

The collection $(I^\omega, J^\omega; A^\omega)_{\omega \in \Omega}$ of pairs of ideals and $C^*$-algebras containing them is \textit{uniformly excisive} if for any $r_0, m_0 \geq 0$ and $\varepsilon > 0$, there are $r \geq r_0$, $m \geq 0$, and $\delta > 0$ such that:

\begin{itemize}
    \item[(i)] for any $\omega \in \Omega$ and any $a \in (A^\omega \otimes \mathcal{K})_{r_0}$ of norm at most $m_0$, there exist elements $b \in (I^\omega \otimes \mathcal{K})_r$ and $c \in (J^\omega \otimes \mathcal{K})_r$ of norm at most $m$ such that $\|a - (b + c)\| < \varepsilon$;
    
    \item[(ii)] for any $\omega \in \Omega$ and any $a \in I^\omega \otimes \mathcal{K} \cap J^\omega \otimes \mathcal{K}$ such that
    \[
    d(a, (I^\omega \otimes \mathcal{K})_{r_0}) < \delta \quad \text{and} \quad d(a, (J^\omega \otimes \mathcal{K})_{r_0}) < \delta,
    \]
    there exists $b \in I^\omega_r \otimes \mathcal{K} \cap J^\omega_r \otimes \mathcal{K}$ such that $\|a - b\| < \varepsilon$.
\end{itemize}
\end{definition}

\begin{theorem}[Proposition 7.7,~\cite{GuentnerWillettYu2024DynamicalComplexity}]\label{mayer-vietoris}
Let $\{(I^{\omega}, J^{\omega}; A^{\omega})\}_{\omega \in \Omega}$ be a uniformly excisive collection, where all algebras and ideals are non-unital. Then, for any $r_0 \geq 0$, there exist $r_1, r_2 \geq r_0$ satisfying the following:

For each $\omega \in \Omega$ and each $x \in K^{r_0,1/8}_*(A^\omega)$, there exists an element
\[
\partial_c (x) \in K^{r_1,1/8}_*(I^{\omega} \cap J^{\omega})
\]
such that if $\partial_c (x)= 0$, then there exist elements
\[
y \in K^{r_2,1/8}_*(I^{\omega}), \quad z \in K^{r_2,1/8}_*(J^{\omega})
\]
such that
\[
x = y + z \quad \text{in } K^{r_2,1/8}_*(A^\omega).
\]
Furthermore, the boundary map $\partial_c$ has the following naturality property:

Let $\{(K^{\theta}, L^{\theta}; B^{\theta})\}_{\theta \in \Theta}$ be another uniformly excisive collection of non-unital filtered $C^*$-algebras and ideals. Assume there exists a map $\pi : \Theta \to \Omega$ and, for each $\theta \in \Theta$, an inclusion
\[
A^{\pi(\theta)} \subset B^{\theta}
\]
such that for each $r \geq 0$:
\[
A^{\pi(\theta)}_r \subset B^{\theta}_r, \quad
I^{\pi(\theta)}_r \subset K^{\theta}_r, \quad
J^{\pi(\theta)}_r \subset L^{\theta}_r.
\]
Let $r_0$ be given, and let $r_1$ be as in the statement above for both uniformly excisive families.
Then the following diagram commutes:
\[
\begin{tikzcd}[column sep=large, row sep=large]
K^{r_0,1/8}_*(A^{\pi(\theta)}) \arrow[r, "\partial_c"] \arrow[d] & K^{r_1,1/8}_*(I^{\pi(\theta)} \cap J^{\pi(\theta)}) \arrow[d] \\
K^{r_0,1/8}_*(B^{\theta}) \arrow[r, "\partial_c"] & K^{r_1,1/8}_*(K^{\theta} \cap L^{\theta}),
\end{tikzcd}
\]
where the vertical maps are induced by subspace inclusions.
\end{theorem}

\section{Relative hyperbolic groups, decomposition complexity and the strategy of the proof}
In this section, we recall some properties of relative hyperbolic groups from \cite{Osin2005AsymptoticDimensionofRelativelyhyperbolicGroups} and outline the strategy of the proof of Theorem~\ref{main}. In this section, we assume that $G$ is a finitely generated group which is hyperbolic relative to a finite family of subgroups $\{H_1,\cdots ,H_n\}$. 

We fix a symmetric generating set $S$ of $G$ and denote by $d_S$ and $d_{S\cup \bigcup H_i}$ the Cayley distance on $G$ with respect to the generating set $S$ and $S\cup \bigcup H_i$, respectively. We denote by $B(n)$ the ball of radius $n$ under the metric $d_{S\cup \bigcup H_i}$ on $G$. 

\begin{lemma}[Proposition 6.8 of \cite{Osin2006RelativelyHyperbolicGroups}]\label{weak relative hyperbolicity}
    In the above setting, the metric space $(G,d_{S\cup \bigcup H_i})$ is hyperbolic.
\end{lemma}

By definition of $B(n)$, we have
\begin{align*}
    B(n)=\left(\bigcup_j B(n-1)H_j\right)\cup B(n-1)S.
\end{align*}
Take a representative $R(n-1)\subset B(n-1)$ from a subset $B(n-1)H_j\subset G/H_j$ of the right cosets so that we have 
\begin{align*}
    B(n-1)H_j=\bigsqcup_{g\in R(n-1)} gH_j.
\end{align*}

The following lemma plays a crucial role in the proof of the twisted Baum--Connes conjecture of $B(n)$. 

\begin{lemma}[Proposition 6.1 of \cite{DadarlatGuentner2007Uniform} and proof of Lemma 3.2 of \cite{Osin2005AsymptoticDimensionofRelativelyhyperbolicGroups}]\label{decomposition of Osin}
    For every $L>0$, there exists $\kappa(L)>0$ such that if
    \begin{align*}
        Y=\{x\in G:d_S(x,B(n-1))\leq \kappa(L)\},
    \end{align*}
    then for each $j$,
    \begin{align*}
        B(n-1)H_j\subset Y\sqcup \left(\bigsqcup_{g\in R(n-1)} gH_j\setminus Y\right)
    \end{align*}
    and the subspaces $\{gH_j\setminus Y\}_{g\in R(n-1)}$ are $L$-separated. 
\end{lemma}
Recall that a family of subspaces $\{X_i\}_i$ of a metric space $(X,d)$ is said to be $L$-separated if $d(X_i, X_j)\geq L$ for $i\neq j$. 

Now, we recall the notion of decomposition of metric spaces from \cite{GuentnerTesseraYu2012ANotionofGeometricComplexity} and \cite{GuentnerTesseraYu2013DiscreteGroupswithFiniteDecompositionComplexity}.

\begin{definition}
    An $r$-decomposition of a metric space $X$ over a metric family $\mathcal{Y}$ is a decomposition
    \begin{align*}
        X=X_0\cup X_1,\quad X_i=\bigsqcup_{j} X_{i,j}
    \end{align*}
    where each $X_{i,j}\in \mathcal{Y}$ and for $i=0,1$, the metric family $\{X_{i,j}\}_j$ is $r$-separated. A metric family $\mathcal{X}$ is $r$-decomposable over $\mathcal{Y}$ if each $X\in \mathcal{X}$ admits an $r$-decomposition over $\mathcal{Y}$. 
\end{definition}

\begin{definition}
    Let $\mathfrak{A}$ be a collection of metric families. A metric family $\mathcal{X}$ is decomposable over $\mathfrak{A}$, if for every $r>0$ there exists $\mathcal{Y}\in \mathfrak{A}$ such that $\mathcal{X}$ admits an $r$-decomposition over $\mathcal{Y}$.
\end{definition}

\begin{example}[Theorem 4.1 of \cite{GuentnerTesseraYu2013DiscreteGroupswithFiniteDecompositionComplexity}]
    Let $\mathfrak{C}$ be the smallest collection of metric families, which contains all uniformly bounded metric families and closed under decomposition (i.e. if $\mathcal{Y}$ decomposes over $\mathfrak{C}$, then $\mathcal{Y}\in \mathfrak{C}$). A metric space $X$ is said to have finite decomposition complexity if the metric family $\{X\}$ is in $\mathfrak{C}$.\\
    If a metric space $X$ has finite asymptotic dimension (not necessarily with bounded geometry), then $X$ has finite decomposition complexity. Especially, the metric family $\{(G,d_{s\cup \bigcup H_i})\}$ has finite decomposition complexity.
\end{example}

Let us clarify what we need to prove to show Theorem~\ref{main}. For this we define a class of metric family.

\begin{definition}\label{class of metric family of uniform cbc}
    A metric family $\mathcal{X}$ is said to satisfy the twisted coarse Baum--Connes conjecture uniformly with respect to any stable coarse algebras if for any  stable coarse $\bigsqcup_{X\in \mathcal{X}}X$-algebra $\Gamma(\bigsqcup_{X\in \mathcal{X}} X,\mathcal{A}\otimes \mathcal{K})$, we have
    \begin{align*}
        \varinjlim_{s} K_*\left(C^*_{L,0}\left(\bigsqcup_{X\in \mathcal{X}} P_s(X);\Gamma\left(\bigsqcup_{X\in \mathcal{X}} X,\mathcal{A}\otimes \mathcal{K}\right)\right)\right)=0.
    \end{align*}
    We denote by $\mathfrak{D}$ the class of metric families $\mathcal{X}$ which consists of subspaces of $G$ which satisfies the twisted coarse Baum--Connes conjecture uniformly with respect to any stable coarse algebras.  
\end{definition}

Since $(G,d_{s\cup \bigcup H_i})$ has finite decomposition complexity, to prove that $\{(G,d_S)\}\in \mathfrak{D}$, it suffices to show the following lemmas:

\begin{lemma}\label{base step}
    Assume that each peripheral subgroup $H_i$ satisfies the twisted coarse Baum--Connes conjecture with respect to any stable coarse algebras. If $\mathcal{X}$ is a metric family consists of subspaces of $(G,d_S)$ whose diameter with respect to the distance $d_{s\cup \bigcup H_i}$ is uniformly bounded, then $\mathcal{X}\in \mathfrak{D}$.
\end{lemma}

\begin{lemma}\label{inductive step}
    If $\mathcal{X}$ decomposes over $\mathfrak{D}$, then $\mathcal{X}$ is in $\mathfrak{D}$.
\end{lemma}

We will prove Lemma~\ref{base step} in Section 5.3 and Lemma~\ref{inductive step} in Section 5.2. \\

Combining Lemmas \ref{base step} and \ref{inductive step}, we obtain our main result.

\begin{theorem}
    Let $G$ be a finitely generated group which is hyperbolic relative to a finite family of subgroups $\{H_1,\cdots ,H_N\}$. 
    Then $G$ satisfies the twisted coarse Baum--Connes conjecture with respect to any stable coarse algebras if and only if each $H_i$ does.
\end{theorem}

\section{Proof of Lemma~\ref{base step} and \ref{inductive step}}
In this section, we prove Lemma \ref{base step} and Lemma \ref{inductive step} to complete the proof of Theorem~\ref{main}. The lemmas are consequences of controlled Mayer-Vietoris sequences, which first appeared in \cite{Yu1998NovikoforFiniteAsymptoticDimension}. The proof is divided into three parts. In the first part, we show that a decomposition of metric spaces induces a decomposition of obstruction algebras so that we can use the controlled Mayer-Vietoris sequence formulated in \cite{GuentnerWillettYu2024DynamicalComplexity}. Second, using the cutting-and-pasting argument, we reduce the problem to the twisted coarse Baum--Connes conjecture for $B(n)$ for each $n$. Finally, we show that any metric family consisting of subspaces of $B(n)$ satisfies the twisted coarse Baum--Connes conjecture with respect to any stable coarse algebras again using the cutting-and-pasting argument.

\subsection{Decomposition}

Fix $s_0\in \mathbb{N}$.
    Let $(X,d)$ be a uniformly locally finite metric space and assume that $\omega=(Y_0,Y_1)$ be a decomposition of $X$. (i.e. $Y_0$ and $Y_1$ are subspaces of $X$ such that $X=Y_0\cup Y_1$.) Let $R_r$ be an increasing function of $r$ such that $R_r> (s_0+1)r$. We decompose $Y_i=\bigsqcup Y_i^{j}$ into $3R_r$-separated subsets.
Then for $i=0,1$ we have
\begin{align*}
    N_r(P_{s_0}(Y_i))=\bigsqcup_j N_r(P_{s_0}(Y_i^{j})),
\end{align*}
where $N_r$ is the $r$-neighborhood in $P_{s_0}(X)$ and the disjoint union is at least $R_r$-separated. We define the new distance $d^{R,r}_i$ on $ N_r(P_{s_0}(Y_i))=\bigsqcup_j N_r(P_{s_0}(Y_i^{j}))$ by 
\begin{align*}
    d_i^{R,r}(x,y)=\left\{\begin{array}{cc}
        d(x,y) & x\in N_r(P_{s_0}(Y_i^{j})), ~ y\in N_r(P_{s_0}(Y_i^{j})) \text{ for some }j   \\
         \infty & \text{otherwise.}
    \end{array} \right.
\end{align*}
We denote by $P_{s_0}(Y_i)^{+r}$ the space $N_r(P_{s_0}(Y_i))$ equipped with the distance $d_i^{R,r}$.

For any coarse $X$-algebra $\Gamma(X,\mathcal{A}\otimes \mathcal{K})$ and a subspace $X'\subset X$, we define a coarse $X'$-algebra $\Gamma(X',\mathcal{A}\otimes \mathcal{K})$ by
\begin{align}\label{coarse alg of subspace}
    \Gamma(X',\mathcal{A}\otimes \mathcal{K}):=\{f|_{X'}:f\in\Gamma(X,\mathcal{A}\otimes \mathcal{K}) \}.
\end{align}
We first define filtered subspaces of $C^*_{L,0}(P_{s_0}(X);\Gamma(X,\mathcal{A}\otimes \mathcal{K}))$ by 
\begin{equation*}
    \begin{aligned}
        B_{r}^{X,R,\omega}&:=C^*_{L,0}(P_{s_0}(Y_0)^{+r};\Gamma(Y_0,\mathcal{A}\otimes\mathcal{K}))_r+C^*_{L,0}(P_{s_0}(Y_1)^{+r};\Gamma(Y_1,\mathcal{A}\otimes\mathcal{K}))_r,\\
        I_{r}^{X,R,\omega}&:=C^*_{L,0}(P_{s_0}(Y_0)^{+r};\Gamma(Y_0,\mathcal{A}\otimes\mathcal{K}))_r,\\
        J_{r}^{X,R,\omega}&:=C^*_{L,0}(P_{s_0}(Y_1)^{+r};\Gamma(Y_1,\mathcal{A}\otimes\mathcal{K}))_r.
    \end{aligned}
\end{equation*}
Then we define $C^*$-subalgebras of $C^*_{L,0}(P_{s_0}(X);\Gamma(X,\mathcal{A}\otimes \mathcal{K}))$:
$$B^{X,R,\omega}:=\overline{\bigcup_{r\geq 0}B_r^{\omega}},~I^{X,R,\omega}:=\overline{\bigcup_{r\geq 0}I_r^{\omega}(X)},~\mbox{and}~I^{X,R,\omega}:=\overline{\bigcup_{r\geq 0}I_r^{\omega}(X)}.$$


Now, we show that $I^{X,R,\omega}$ and $J^{X,R,\omega}$ are filtered ideals of $B^{X,R,\omega}$ (see Definition~\ref{filtered ideal}). It suffices to show that for any $S\in C^*(P_{s_0}(Y_1)^{+r_1};\Gamma(Y_1,\mathcal{A}\otimes\mathcal{K}))_{r_1}$ and $T\in C^*(P_{s_0}(Y_0)^{+r_2};\Gamma(Y_0,\mathcal{A}\otimes\mathcal{K}))_{r_2}$, the product $ST$ is an element in $C^*(P_{s_0}(Y_0)^{+r_1+r_2};\Gamma(Y_0,\mathcal{A}\otimes\mathcal{K}))_{r_1+r_2}$. We only need to verify the condition (1) of Definition~\ref{definition of twisted roe algebra}. Let us fix a Borel cover $\{B_x\}_{x\in Y_0}$ of $P_{s_0}(Y_0)^{+r_1+r_2}$ satisfying the four conditions in Lemma~\ref{borel cover}. For any partial translation $v$ on $Y_0$, we show that the function
\begin{align*}
    (ST)^v:Y_0\to \mathcal{A}\otimes \mathcal{K}, ~ x\mapsto \chi_{B_{v(x)}} (ST)\chi_{B_x} =\sum_y \left(\chi_{B_{v(x)}} S\chi_{B_y}\right)\left( \chi_{B_{y}} T\chi_{B_x}\right)
\end{align*}
is in $\Gamma(Y_0,\mathcal{A}\otimes\mathcal{K})$. Both $S$ and $T$ can be decomposed into finite sums of the operators of the form $f\cdot v$, where $f$ is an $\ell^{\infty}$ function and $v$ is a partial translation. 
By further decomposing partial translations, we may assume that for each $x\in Y_0$, there is at most one $y$ such that $\chi_{B_{v(x)}} S\chi_{B_y}\neq 0$ and $\chi_{B_{y}} T\chi_{B_x}\neq 0$. We denote such $y$ by $\alpha(x)$ and denote the assignment $v(x)\to y$ by $\beta$, and then $\alpha$ and $\beta$ are partial translations on $Y_0$. In this case, we have that
\begin{align*}
    (ST)^v(x)=\left(\chi_{B_{v(x)}} S\chi_{B_{\beta v^{-1}(x)}}\right)\left( \chi_{B_{\alpha(x)}} T\chi_{B_x}\right).
\end{align*}

By the definition of coarse algebras for subspaces \eqref{coarse alg of subspace}, there are $\varphi,\psi\in \Gamma(Y_0,\mathcal{A}\otimes\mathcal{K})$ such that $\varphi(x)=\chi_{B_{v(x)}} S\chi_{B_{\beta v^{-1}(x)}}$ and $\psi(x)=\chi_{B_{\alpha(x)}} T\chi_{B_x}$ for all $x\in Y_0$. Therefore the restriction of $\varphi\psi$ to $Y_0$ is also an element in $\Gamma(Y_0,\mathcal{A}\otimes\mathcal{K})$. 

\begin{lemma}\label{excisive}
    The family of triples $(I^{X,R,\omega},J^{X,R,\omega};B^{X,R,\omega})_{X,R,\omega}$ is uniformly excisive (Definition~\ref{definition of excisive}), where $X$ runs all metric space with bounded geometry, $\omega$ runs over all decompositions $\omega=(Y_0,Y_1)$ of $ X$ and $R$ runs over all increasing functions of $r$ satisfying $R_r>(3s_0+1)r$.
\end{lemma}

\begin{proof}
    For given $m_0, \varepsilon, r_0$ in Definition~\ref{definition of excisive}, we prove that $m=4m_0$, $\delta=\frac{1}{2}\varepsilon$ and $r=r_0+1$ work.

First, we prove the first condition of Definition~\ref{definition of excisive}. Note that 
$$P_{s_0}(X)\subset N_{1}(P_{s_0}(Y_0)) \cup N_{1}(P_{s_0}(Y_1)).$$ 
We define 
\begin{align*}
    Z_0&:=N_{r_0}(P_{s_0}(Y_0)) \setminus N_{r_0}(P_{s_0}(Y_1)),\\
    Z_1&:=N_{r_0}(P_{s_0}(Y_1)) \setminus N_{r_0}(P_{s_0}(Y_0)),\\
    Z_{0,1}&:=N_{r_0}(P_{s_0}(Y_0)) \cap N_{r_0}(P_{s_0}(Y_1)).
\end{align*}
Note that $d(Z_0, Z_1)> r_0$. For any $A\in B_{r_0}^{X,R,\omega}$, we define
    \begin{align*}
        A_0&:=\chi_{Z_0} A \chi_{Z_0} + \chi_{Z_0} A \chi_{Z_{0,1}}+ \chi_{Z_{0,1}} A \chi_{Z_0} + \chi_{Z_{0,1}} A \chi_{Z_{0,1}},\\
        A_1&:=\chi_{Z_1} A \chi_{Z_1} + \chi_{Z_1} A \chi_{Z_{0,1}}+ \chi_{Z_{0,1}} A \chi_{Z_1}.
    \end{align*}
    Since $A$ has propagation less than $r_0$, we have $A=A_0+A_1$. In addition, $A_0$ and $A_1$ have propagation less than $r_0$, $A_0$ is supported on $N_{r_0}(P_{s_0}(Y_0))\times N_{r_0}(P_{s_0}(Y_0))$ and $A_1$ is supported on $N_{r_0}(P_{s_0}(Y_1))\times N_{r_0}(P_{s_0}(Y_1))$. Moreover we have
    \begin{align*}
        \|A_0\|\leq 4\|a\|,~\|A_1\|\leq 3\|a\|.
    \end{align*}
    For any partial translation $v$ on $Y_0$, the functions $(A_0)^v$ is in  $\Gamma(Y_0,\mathcal{A}\otimes \mathcal{K})$ by the definition of the coarse algebra \eqref{coarse alg of subspace} since it is a restriction of $A^v$.  Therefore $A_0\in I_{r_0}^{\omega}$. Similarly, we have $A_1\in J_{r_0}^{\omega}$.
    
    Next, we prove the second condition. Let $\{\mu_0, \mu_1\}$ be a partition of unity subordinate to the cover $N_{1}(P_{s_0}(Y_0)) \cup N_{1}(P_{s_0}(Y_1))$. Let $A\in I^{X,R,\omega}\otimes \mathcal{K}\cap J^{X,R,\omega}\otimes \mathcal{K}$ be such that $\|A-A_0\|\leq \delta$ and $\|A-A_1\|\leq \delta$ for some $A_0\in \left(I^{X,R,\omega}\right)_{r_0}$ and $A_1\in \left(J^{X,R,\omega}\right)_{r_0}$. Define
    \begin{align*}
        B:=A_0\mu_1+A_1\mu_0.
    \end{align*}
    Then since the multiplication by $\mu_1$ does not increase the support, 
    \begin{align*}
        \supp{(A_0\mu_1)}\subset \supp(A_0)\subset N_{r_0}(P_{s_0}(Y_0))\times N_{r_0}(P_{s_0}(Y_0)).
    \end{align*}
   Since $\mu_1$ is supported on $N_1(P_{s_0}(Y_1))$, and $A_0$ has propagation less than $r_0$, we obtain that
    \begin{align*}
        \supp{(A_0\mu_1)}\subset N_{r_0+1}(P_{s_0}(Y_1)) \times N_1(P_{s_0}(Y_1))\subset N_{r_0+1}(P_{s_0}(Y_1)) \times N_{r_0+1}(P_{s_0}(Y_1)) 
    \end{align*}
    By arguing similarly for the support of $A_1\mu_0$, we can show that 
    $$\supp(B)\subset \left(N_{r_0+1}(P_{s_0}(Y_0))\cap N_{r_0+1}(P_{s_0}(Y_1))\right) \times \left(N_{r_0+1}(P_{s_0}(Y_0))\cap N_{r_0+1}(P_{s_0}(Y_1))\right).$$
    Also, we have 
    $$\|A-B\|=\|(A-A_0)\mu_1+(A-A_1)\mu_0\|\leq 2 \varepsilon=\delta.$$
    For any partial translation $v$ on $Y_i$, the functions $(A_0\mu_1)^v$ and $(A_1\mu_0)^v$ is in  $\Gamma(Y_i,\mathcal{A}\otimes \mathcal{K})$ and so we have $B\in I^{X,R,\omega}_{r_0+1}\cap J^{X,R,\omega}_{r_0+1}$.
Therefore, the family of triples $(I^{X,R,\omega},J^{X,R,\omega};B^{X,R,\omega})_{X,R,\omega}$ is uniformly excisive.   
\end{proof}
Moreover, for $s\geq s_0$, we take a net 
$$Z_r:= P_s\left(\bigsqcup_j Y_0^{(j)}\right)^{+r}\cap P_{s}\left(\bigsqcup_j Y_1^{(j)}\right)^{+r}\cap X$$ 
of $P_s\left(\bigsqcup_j Y_0^{(j)}\right)^{+r}\cap P_{s}\left(\bigsqcup_j Y_1^{(j)}\right)^{+r}$. Then we define 
\begin{equation*}
    \begin{aligned}
        B_r^{X,R,\omega,s}&:=C^*_{L,0}(P_{s_0}(Y_0)^{+r};\Gamma(Y_0,\mathcal{A}\otimes \mathcal{K}))_r+C^*_{L,0}(P_{s_0}(Y_1)^{+r};\Gamma(Y_1,\mathcal{A}\otimes \mathcal{K}))_{r}\\
        &\quad \quad \quad \quad +C^*_{L,0}\left(P_s\left(\bigsqcup_j Y_0^{(j)}\right)^{+r}\cap P_{s}\left(\bigsqcup_j Y_1^{(j)}\right)^{+r};\Gamma(Z_r,\mathcal{A}\otimes \mathcal{K})\right)_{sr},\\
        I_r^{X,R,\omega,s}&:=C^*_{L,0}(P_{s_0}(Y_0)^{+r};\Gamma(Y_0,\mathcal{A}\otimes \mathcal{K}))_{r}\\
        &\quad \quad \quad \quad +C^*_{L,0}\left(P_s\left(\bigsqcup_j Y_0^{(j)}\right)^{+r}\cap P_{s}\left(\bigsqcup_j Y_1^{(j)}\right)^{+r};\Gamma(Z_r,\mathcal{A}\otimes \mathcal{K})\right)_{sr},\\
        J_r^{X,R,\omega,s}&:=C^*_{L,0}(P_{s_0}(Y_1)^{+r};\Gamma(Y_1,\mathcal{A}\otimes \mathcal{K}))_{r}\\
        &\quad \quad \quad \quad +C^*_{L,0}\left(P_s\left(\bigsqcup_j Y_0^{(j)}\right)^{+r}\cap P_{s}\left(\bigsqcup_j Y_1^{(j)}\right)^{+r};\Gamma(Z_r,\mathcal{A}\otimes \mathcal{K})\right)_{sr},\\
    \end{aligned}  
\end{equation*}
where the Rips complex $P_{s}\left(\bigsqcup_j Y_i^{(j)}\right)$ is taken in terms of the separated coarse disjoint union $\bigsqcup_j Y_i^{(j)}$ and the $r$-neighborhood $P_{k}\left(\bigsqcup_j Y_i^{(j)}\right)^{+r}$ is taken relative to $P_{s_0}(X)$, i.e.
\begin{equation}\label{relative expansion}
    P_{s}\left(\bigsqcup_j Y_0^{(j)}\right)^{+r}=P_{s}\left(\bigsqcup_j Y_0^{(j)}\right)\cup \left\{p\in P_{s_0}(X);d\left(p,P_s\left(\bigsqcup_j Y_0^{(j)}\right)\cap P_{s_0}(X)\right)\leq r\right\}.
\end{equation} 
Then define a $C^*$-algebras $B^{X,R,\omega,s}:=\overline{\bigcup_{r\geq 0}B_r^{X,R,\omega,s}}$ and two ideals $I^{X,R,\omega,s}:=\overline{\bigcup_{r\geq 0}I_r^{X,R,\omega,s}}$ and $J^{X,R,\omega,s}:=\overline{\bigcup_{r\geq 0}J_r^{X,R,\omega,s}}$ of $B^{X,R,\omega,s}$. 

\begin{lemma}\label{uniformly excisive s}
    The family $\left(I^{X,R,\omega,s},J^{X,R,\omega,s};B^{X,R,\omega,s}\right)_{X,R,\omega,s}$ is uniformly excisive, where $X$ runs over all metric spaces with bounded geometry, $R$ runs over all increasing functions of $r$ satisfying $R_r> 3(k_0+1)r$, $\omega$ runs over all decompositions $\omega=(Y_0,Y_1)$ of $X$ and $s$ runs over all numbers $s\geq s_0$.
\end{lemma}

\begin{proof}
    This follows from the same arguments in the proof of Lemma \ref{excisive}.
\end{proof}

\subsection{Reduction to $(B(n),d_S)$}

In this section, we prove Lemma \ref{inductive step}.
    Let $x\in K_{*}^{r_0,\frac{1}{8}}(C^*_{L,0}( P_{s_0}(\bigsqcup X);\Gamma(\bigsqcup X,\mathcal{A}\otimes \mathcal{K})) $ be any element for any $r_0$ and $s_0$, where $\Gamma\left(\bigsqcup X,\mathcal{A}\otimes \mathcal{K}\right)$ is any stable coarse $\left(\bigsqcup_{X\in \mathcal{X}} X\right)$-algebra. It suffices to show that there exists $s\geq s_0$ such that $x=0$ in
    \begin{align*}
         K_{*}^{s,\frac{1}{8}}\left(C^*_{L,0}\left( P_{s}\left(\bigsqcup X\right);\Gamma\left(\bigsqcup X,\mathcal{A}\otimes \mathcal{K}\right)\right)\right).
    \end{align*}
We denote $\tilde{X}=\bigsqcup_{X\in \mathcal{X}} X$. For any subset $X'\subset \tilde{X}$, recall that $\Gamma(X',\mathcal{A}\otimes \mathcal{K})$ is the coarse $X'$-algebra defined in \eqref{coarse alg of subspace}. We regard the $r_2$ of Lemma~\ref{mayer-vietoris} for the uniformly excisive pair of Lemma~\ref{uniformly excisive s} as a function of $r_0$ and define a function $R:r_0\mapsto (3s_0+1)r_2$.
    Decompose
    \begin{align*}
    X=\bigsqcup_j Y_0^{(j)}(X) \cup \bigsqcup_j Y_1^{(j)}(X)
\end{align*}
such that the family $\{Y_i^{(j)}(X)\}_j$ is $3R_{r_0}$-separated and each $Y_i^{(j)}(X)$ is in $\mathcal{Y}$ for $i=0,1$ for some $\mathcal{Y}\in \mathfrak{D}$. Then this gives a decomposition 
\begin{align*}
    \tilde{X}= \bigsqcup_{j,X} Y_0^{(j)}(X) \cup \bigsqcup_{j,X} Y_1^{(j)}(X).
\end{align*}
We denote this decomposition of $\tilde{X}$ by $\omega=\left(\bigsqcup_{j,X} Y_0^{(j)}(X) , \bigsqcup_{j,X} Y_1^{(j)}(X)\right)$ and denote $Y_i:=\bigsqcup_{j,X} Y_i^{(j)}(X)$ where the distance between distinct components is taken to be infinite.
Note that $C^*_{L,0}\left(P_{s_0}(\bigsqcup X);\Gamma(\bigsqcup X,\mathcal{A}\otimes \mathcal{K})\right)_{r_0}$ is contained in 
    \begin{align*}
        C^*_{L,0}\left(P_{s_0}(Y_0)^{+r_0};\Gamma\left(Y_0,\mathcal{A}\otimes \mathcal{K}\right)\right)_{r_0} + C^*_{L,0}\left(P_{s_0}(Y_1)^{+r_0};\Gamma\left(Y_1,\mathcal{A}\otimes \mathcal{K}\right)\right)_{r_0} \subset B_{r_0}^{\tilde{X},R,\omega}.
    \end{align*}
Therefore, $x$ can be regarded as an element in $K^{\frac{1}{8},r_0}(B^{\tilde{X},R,\omega})$.
There exists $r_1\geq r_0$ and the boundary maps $\partial$, which fits in the commutative diagram below for every $s_1\geq s_0$:
\[ \begin{tikzcd}
K_*^{r_0,\frac{1}{8}}(B^{\tilde{X},R,\omega}) \arrow{r}{\partial} \arrow{d} &K_*^{r_1,\frac{1}{8}}(I^{\tilde{X},R,\omega}\cap J^{\tilde{X},R,\omega}) \arrow{d} \\
K_*^{r_0,\frac{1}{8}}(B^{\tilde{X},R,\omega,s_1}) \arrow{r}{\partial}& K_*^{(s_0+1)r_1,\frac{1}{8}}(I^{\tilde{X},R,\omega,s_1}\cap J^{\tilde{X},R,\omega,s_1}).
\end{tikzcd}
\]
Note that since we have
\begin{align*}
    P_{s_0}(N_{r}(Y_i(X))) \subset P_{s_0}(Y_i(X))^{+r} \subset P_{s_0} (N_{(s_0+1)r}(Y_i(X))),
\end{align*}
the right vertical map can be identified with the compositions of following natural maps
\begin{align*}
    K_{*}&^{r_1,\frac{1}{8}}\left(C^*_{L,0}\left(P_{s_0}\left(\bigsqcup_{j,X} Y_0^{(j)}(X)\right)^{+r}\cap P_{s_0}\left(\bigsqcup_{j,X} Y_1^{(j)}(X)\right)^{+r};\Gamma(Z_r,\mathcal{A}\otimes \mathcal{K})\right)\right) \\
    &\longrightarrow K_{*}^{r_1,\frac{1}{8}}\left(C^*_{L,0}\left(P_{s_0}\left(\bigsqcup_{j,j',X} N_{(s_0+1)r_0}(Y_0^{(j)}(X))\cap N_{(s_0+1)r_0} (Y_1^{(j')}(X))\right);\Gamma(Z_r,\mathcal{A}\otimes \mathcal{K})\right)\right) \\
    &\longrightarrow K_{*}^{r_1s_1,\frac{1}{8}}\left(C^*_{L,0}\left(P_{s_1}\left(\bigsqcup_{j,j',X} N_{(s_0+1)r_0}(Y_0^{(j)}(X))\cap N_{(s_0+1)r_0} (Y_1^{(j')}(X))\right);\Gamma(Z_r,\mathcal{A}\otimes \mathcal{K})\right)\right) \\
    &\longrightarrow K_*^{(s_0+1)r_1,\frac{1}{8}}\left(I^{\tilde{X},R,\omega,s_1}\cap J^{\tilde{X},R,\omega,s_1}\right).
\end{align*}
By the coarse invariance and the permanence under taking subspace of the twisted coarse Baum--Connes conjecture with respect to any stable coarse algebras, $\partial x$ is sent to $0$ for sufficiently large $s_1$. By the controlled-exactness of the Mayer-Vietoris sequence, there exists $r_2$ and
\begin{align*}
    y\in K_*^{r_2,\frac{1}{8}}(I^{\tilde{X},R,\omega,s_1}), ~  z\in K_*^{r_2,\frac{1}{8}}(J^{\tilde{X},R,\omega,s_1})
\end{align*}
such that $x=y+z$ in $K_*^{r_2,\frac{1}{8}}(B^{\tilde{X},R,\omega,s_1})$. 
Since we have
\begin{align*}
    I^{X,R,\omega,s_1}_{r_2}
    &\subset C^*_{L,0}\left(P_{s_1}\left(\bigsqcup_{j,X} Y_0^{(j)}(X)\right)^{+r_2};\Gamma(Y_0,\mathcal{A}\otimes\mathcal{K})\right)_{s_1r_2}\\
    &\subset C^*_{L,0}\left(P_{s_1}\left(\bigsqcup_j N_{(s_0+1)r_2}(Y_0^{(j)}(X))\right);\Gamma(Y_0,\mathcal{A}\otimes\mathcal{K})\right)_{s_1r_2}\\
    &=C^*_{L,0}\left(\bigsqcup_j P_{s_1}\left( N_{(s_0+1)r_2}(Y_0^{(j)}(X))\right);\Gamma(Y_0,\mathcal{A}\otimes\mathcal{K})\right)_{s_1r_2},
\end{align*} 
by the definition \eqref{relative expansion} of the relative $r_2$-neighborhood of $P_{s_1}\left(\bigsqcup_j Y_0^{(j)}(X)\right)$ and by the assumption that each $Y_0^{(j)}(X)\in \mathcal{Y}$, $y=0$ in $K_*^{r,\frac{1}{8}}(C^*_{L,0}(P_s(X));\Gamma(\bigsqcup X,\mathcal{A}\otimes \mathcal{K}))$ for sufficiently large $r$ and $s$. Similarly, we have $z=0$ in $K_*^{r,\frac{1}{8}}(C^*_{L,0}(P_s(X));\Gamma(\bigsqcup X,\mathcal{A}\otimes \mathcal{K}))$.
Therefore, we have
\begin{align*}
    x=y+z=0 \in K_*^{s,\frac{1}{8}}\left(C^*_{L,0}\left( P_{s}\left(\bigsqcup X\right);\Gamma\left(\bigsqcup  X;\mathcal{A}\otimes \mathcal{K})\right)\right)\right).
\qedhere
\end{align*}

\subsection{The twisted coarse Baum--Connes conjecture for $B(n)$}

In this subsection, we prove Lemma~\ref{base step}.

\begin{lemma}\label{family of subspaces}
    Let $\{Y_j\}_{j=1}^{\infty}$ be a sequence of subspaces of $X$. If $X$ satisfies the twisted coarse Baum--Connes conjecture with respect to any stable coarse algebras, then $\{Y_j\}$ satisfies the twisted coarse Baum--Connes conjecture uniformly with respect to any stable coarse algebras.
\end{lemma}

\begin{proof}
    This follows from Corollary~\ref{permanence for subsp.} and Lemma~\ref{uniform coarse BaumConnes of subspaces}.
\end{proof}

We record a useful corollary, which is very specific to the twisted coarse Baum--Connes conjecture. 
\begin{corollary}\label{union}
    Let $X\subset Y_1,Y_2$. If $Y_1$ and $Y_2$ satisfy the twisted coarse Baum--Connes conjecture with respect to any stable coarse algebras, then so does $Y_1\cup Y_2$.
\end{corollary}
\begin{proof}
    This follows from Lemma \ref{inductive step}.
\end{proof}

\begin{corollary}\label{B(n)}
    Assume $G$ is hyperbolic relative to $\{H_i\}_{i=1}^N$.
    If every $H_i$ satisfies the twisted coarse Baum--Connes conjecture with any stable coarse algebras, then so does $B(n)$.
\end{corollary}

\begin{proof}
    Induction on $n$.  When $n=1$, use Lemma~\ref{family of subspaces} and Corollary~\ref{union} to
    \begin{align*}
        B(1)=S \cup \left(\bigcup_i H_i\right).
    \end{align*}
    Assume the conclusion holds for $1,2,\cdots,n-1$. Denote 
    \begin{align*}
        B(n;i)= B(n-1)\cdot S\cup  B(n-1) \cdot H_i.
    \end{align*}
    Since $B(n)=\bigcup_{i=1}^n B(n;i)$, it suffices to show that $B(n;i)$ satisfies the twisted coarse Baum--Connes conjecture with any stable coarse algebras, again by Corollary~\ref{union}. But this follows from Lemma~\ref{decomposition of Osin}, Lemma~\ref{inductive step} and Corollary~\ref{union}.
\end{proof}

Now Lemma~\ref{base step} follows from Lemma~\ref{family of subspaces} and Corollary~\ref{B(n)}.

\section*{Acknowledgements} 

We would like to thank Prof. Guoliang Yu for his comments and discussions on this topic.

\bibliographystyle{alpha}
\bibliography{ref}
\end{document}